%% file: delta_invariant_of_Fano_weighted_hypersurfaces.tex
\documentclass[12pt]{amsart}
\input{macros}

\usepackage[colorlinks=true,pagebackref,hyperindex]{hyperref}
\usepackage{amsmath}
\usepackage{amssymb}
\usepackage{amsthm}
\usepackage[all,cmtip]{xy}

\newcommand{\bA}{\ensuremath{\mathbb{A}}}

\newcommand{\bC}{\ensuremath{\mathbb{C}}}

\newcommand{\bN}{\ensuremath{\mathbb{N}}}

\newcommand{\bP}{\ensuremath{\mathbb{P}}}
\newcommand{\bQ}{\ensuremath{\mathbb{Q}}}
\newcommand{\bR}{\ensuremath{\mathbb{R}}}

\newcommand{\bZ}{\ensuremath{\mathbb{Z}}}
\newcommand{\bfa}{\ensuremath{\mathbf{a}}}

\newcommand{\cF}{\ensuremath{\mathcal{F}}}

\newcommand{\cI}{\ensuremath{\mathcal{I}}}

\newcommand{\cL}{\ensuremath{\mathcal{L}}}

\newcommand{\cO}{\ensuremath{\mathcal{O}}}

\DeclareMathOperator{\codim}{codim}
\DeclareMathOperator{\Image}{Im}
\DeclareMathOperator{\Inte}{Int}
\DeclareMathOperator{\Val}{Val}
\DeclareMathOperator{\NQS}{NQS}

\newcommand{\dps}{\displaystyle}
\newcommand{\bul}{\bullet}

\title{Delta invariants of weighted hypersurfaces}
\date{\today}

\author{Taro Sano}
\email{tarosano@math.kobe-u.ac.jp} 
\address{Department of Mathematics, Graduate School of Science, 
Kobe university, 
1-1, Rokkodai, Nada-ku, Kobe 657-8501, Japan}

\author{Luca Tasin}
\email{luca.tasin@unimi.it}
\address{Dipartimento di Matematica F.\ Enriques, Universit\`a degli Studi di Milano, Via Cesare Saldini 50, 20133 Milano, Italy}

\begin{document}



\maketitle
\begin{abstract}
We give a lower bound for the delta invariant of the fundamental divisor of a quasi-smooth weighted hypersurface.
As a consequence, we prove K-stability of a large class of quasi-smooth Fano hypersurfaces of index 1 and of all smooth Fano weighted hypersurfaces of index 1 and 2.
The proofs are based on the Abban--Zhuang method and on the study of linear systems on flags of weighted hypersurfaces.	
\end{abstract}




\section{Introduction}

As predicted by the Yau--Tian--Donaldson conjecture and confirmed by the recent breakthroughs \cite{CDS, MR3352459, LXZ22}, the existence of a K\"ahler-Eintein metric on a klt Fano variety $X$ is equivalent to the K-polystability of $X$. 
An important feature of K-stability is that it can be characterised by means of delta invariants associated to the anticanonical divisor $-K_X$ (\cite{ MR3956698, MR3896135, MR4067358}). 
More generally, if $X$ is a klt variety and $L$ is a big line bundle on $X$, one can define the $\delta$-invariant of $(X,L)$ via log canonical thresholds of basis type divisors $\bQ$-linearly equivalent to $L$, see Section \ref{s:thresholds} for details. A klt Fano variety $X$ is K-stable if and only if $\delta(-K_X) >1$. For an overview on the subject see \cite{Xu21} and \cite{Xubook}.



A challenging conjecture predicts that any smooth Fano hypersurface $ X \subset \bP^{n+1}$ of degree at least $3$ is K-stable. This is known to be true if the Fano index of $X$ is $1$ \cite{Pukhlikov98, Cheltsov01}, \cite[Corollary 1.5]{MR3936640}, and $2$ \cite{AZ22} or if $X$ is general. The latter follows from the openness of K-stability and the fact that Fermat hypersurfaces of degree at least 3 are K-stable (see \cite{MR894378}, \cite{Tian00}, \cite{MR2212883},  \cite{MR4309493} for historical and conceptual foundation). It is also known for cubic fourfolds \cite{Liu22} and if the dimension of $X$ is at least $\iota_X^3$ \cite[Theorem B]{AZ23}, where $\iota_X$ is the Fano index of $X$.

In this paper, we look at hypersurfaces in the weighted projective space, since these provide a large class of interesting (singular) Fano varieties.
Our main result is the following.

\begin{theorem}\label{thm:main_intro}[=Theorem \ref{thm:main}]
Let $X=X_d \subset \bP(a_0, \ldots, a_{n+1})$ be a well-formed quasi-smooth weighted hypersurface of degree $d$ which is not a linear cone. 
Assume that there is $r$  such that $a_r >1$ and $a_r | d$.
Then 
$$
\delta(\cO_X(1)) \ge \frac{(n+1)a_{r}}{d}.
$$

Moreover, if  $X$ is Fano of index $\iota_X: = \sum_{i=0}^{n+1} a_i -d$ and   $\frac{(n+1)a_{r}}{d} \ge \iota_X$, then $\delta(-K_X) \ge 1$ and if $n \ge 3$, then $X$ is K-stable. 
\end{theorem}


As an explicit application, one can use Theorem \ref{thm:main_intro} in dimension 3 and 4,  see Subsection \ref{s:low} for details.
Very little is known of K-stability of Fano weighted hypersurfaces in higher dimension (see \cite{Johnson-Kollar, MR4056840}). 
Assuming that $a_i |d$ for all $i$, in \cite[Corollary 1.5]{ST21}, we showed that general Fano hypersurfaces of degree $d$ and $\iota_X < \dim X$ are K-stable. This has been then applied in \cite{LST} to construct infinitely many Sasaki--Eintein metrics on odd-dimensional spheres. It usually takes some work to generalize results for general members to those for all members. 

As a consequence of Theorem \ref{thm:main_intro}, we get the following. 

\begin{corollary}\label{c:quasi-smooth}
	Let $X=X_d \subset \bP(a_0, \ldots, a_{n+1})$ be a well-formed quasi-smooth Fano weighted hypersurface of degree $d$ and index $1$ such that $a_0  \le a_1 \le \ldots  \le a_{n+1}$. Assume that $a_{n+1} | d$. Then 
	$$
	\delta(-K_X) > 1,
	$$
	so $X$ is K-stable. 
\end{corollary}

This establishes the K-stability of much larger classes of Fano hypersurfaces than the ones treated in \cite[Theorems 1.1, 1.4]{ST21}. 

If $X$ is a smooth Fano hypersurface, then we get the K-stability if the index is 1,2 or 3 (with two exceptions in the index $3$ case): 

\begin{corollary} \label{c:smooth}
	Let $X=X_d \subset \bP(a_0, \ldots, a_{n+1})$ be a well-formed smooth Fano weighted hypersurface of degree $d$ and index $\iota_X$.  Assume $a_0 \le a_1 \le \ldots \le a_{n+1}$ and $a_{n+1} >1$. Then we have the following. 
	\begin{enumerate}[(i)]
		\item If $\iota_X=1$, then
		$$ 	
		\delta(-K_X) \ge \frac{a_{n+1}}{2},
		$$
		and if the equality holds, then $X=X_{2a_{n+1}} \subset \bP(1,\ldots,1,2,a_{n+1})$.
		In any case, $\delta(-K_X)>1$ and so $X$ is K-stable.
		
		\item If $\iota_X=2$, then $\delta(-K_X)>1$.
		\item If $\iota_X =3$, then $\delta(-K_X) > 1$ except possibly in the following cases: $X=X_{2k} \subset \bP(1,\ldots,1, 2)$ or $X=X_{6k} \subset \bP(1,\ldots, 1,2,3)$ with $k$ positive integer.
	\end{enumerate}		
\end{corollary}	

The index 1 case realizes one of our original motivations in \cite{ST21} without solving \cite[Question 1.3]{ST21}. 
This exhibits the effectivity of the Abban--Zhuang method for the K-stability of explicit Fano varieties.

\subsection{Low dimensional cases} \label{s:low}
It is classically known (proven first by Tian \cite{Tian90}) that all smooth del Pezzo surfaces are K-polystable except for the blowups of $\bP^2$ in one or two points. The problem of K-stability for smooth Fano threefolds has been intensively studied, see \cite{Calabiproblem} for a recent compendium. 
In this subsection we focus on the case of weighted hypersurfaces of dimension at most 4.

Quasi-smooth del Pezzo surfaces $X \subset \bP(a_0,\ldots,a_3)$ of index 1 have been classified in \cite{Johnson-Kollar2} and the existence of K\"{a}hler--Einstein metrics has been proven in all cases (cf.\ Ibid, \cite{MR1926877}, \cite{CJS10}, \cite{MR4225810}). Note that there exist unstable examples with index 2 \cite{KW21}. It is worth recalling the Johnson--Koll\'ar criterion \cite{Johnson-Kollar2} where the authors showed that if $X_d \subset \bP(a_0,\ldots,a_{n+1})$ is a Fano weighted hypersurface of  index 1 with $a_0 \le \ldots \le a_{n+1}$ such that 
\[
d  < \frac{n+1}{n}a_0a_1,
\]
then $X_d$ admits a K\"ahler--Eintein metric. Such criterion is useful when all weights are relatively big. 

There are 95 families of terminal quasi-smooth Fano 3-fold hypersurfaces $X \subset \bP(a_0,\ldots,a_4)$ of index 1, see \cite{Johnson-Kollar,CCC11}, \cite{Fletcher00}. By \cite{Cheltsov08,MR2499678}, the general element of each family is K-stable. Using the relation with birational superrigidity (see \cite{SZ19,KOW18,KOW23}), the K-stabily of all members of 75 families have been proven, as reported in \cite[Table 7]{KOW23}.  

Quasi-smooth Fano 4-fold hypersurfaces  $X \subset \mathbb P(a_0,\ldots, a_5)$ of index 1 are classified in \cite{BK16}, see \cite{BK02} for a complete list. There are 11618 families with terminal members.

Theorem \ref{thm:main_intro} can be directly applied to get the following consequence.

\begin{corollary}\label{c:low}
There are 82 out of 95 (resp. 7483 out of 11618) families of quasi-smooth terminal Fano $3$-fold hypersurfaces $X_d \subset \bP(a_0,\ldots,a_4)$ (resp. $4$-folds  $X_d \subset \bP(a_0,\ldots,a_5)$) of index $1$ for which there exists $r$ such that $a_r >1$, $a_r |d$ and 
$$
\frac{4 a_{r}}{d} \ge 1 \ \ \left( \text{resp. } \frac{5a_r}{d} \ge 1 \right). 
$$

Hence, all members of the above families are K-stable.
\end{corollary}	

\begin{proof}
One can easily check the lists by hand for 3-folds and with a computer for 4-folds using the database \cite{BK02}.  	
The 3-folds families for which Theorem \ref{thm:main_intro} does not apply are the following (the numbering is the one of \cite[Table 7]{KOW23}): No.2, 5, 12, 13, 20, 23, 25, 33, 38, 40, 58, 61 and 76.
\end{proof}

The following table collects the 9 families for which K-stability of all members were not known before and it is given by Corollary \ref{c:low}. 

\medskip 

	\begin{center}
	\begin{tabular}{ |c|c| } 
		\hline
		$No. $ & $X_d \subset \bP(a_0,a_1,a_2,a_3,a_4)$  \\  \hline
		$4$ & $X_6 \subset \bP(1,1,1,2,2)$  \\ 
		$7$ & $X_8 \subset \bP(1,1,2,2,3)$  \\ 
		$9$ & $X_9 \subset \bP(1,1,2,3,3)$  \\ 
		$18$ & $X_{12} \subset \bP(1,2,2,3,5)$  \\ 
		$24$ & $X_{15} \subset \bP(1,1,2,5,7)$  \\ 
		$31$ & $X_{16} \subset \bP(1,1,4,5,6)$  \\ 
		$32$ & $X_{16} \subset \bP(1,2,3,4,7)$  \\ 
		$43$ & $X_{20} \subset \bP(1,2,4,5,9)$  \\ 
		$46$ & $X_{21} \subset \bP(1,1,3,7,10)$  \\ 
		\hline
	\end{tabular}
\end{center}
For example, we obtain the K-stability of $X_{12}$ in No.18 by $\frac{4\cdot 3}{12} \ge 1$. 

In conclusion, the families for which K-stability is yet to be determined are No. $2,5,12,13,20,23,25,33,38,40$ and 58 in \cite[Table 7]{KOW23} (since No.61 and 76 are  established in \cite{KOW23}). 

\begin{remark}
Soon after the appearance of the present manuscript, Campo and Okada \cite{CO} established the K-stability in the remaining cases of the 95 families. 
\end{remark}



\subsection{Structure of the paper and the ingredients}\label{ss:structure}

Let $X$ be a variety and $x \in X$ be a point. One of the main achievement of the Abban--Zhuang method is that it gives a bound on the local delta invariant $\delta_x(L)$ of a line bundle $L$ on $X$ if one can construct a suitable flag of subvarieties. 
Roughly speaking, the smaller the degree of $L$, the better the bound (cf. \cite[Lemma 4.4]{AZ22} or Lemma \ref{l:curve}). 
A main technical ingredient is the adjunction type inequality as \cite[Theorem 3.2]{AZ22} (cf. Theorem \ref{thm:AZadjunctionineq}) which compares the delta invariant $\delta_x(L)$ and 
the delta invariant $\delta_x(H, W_{\bullet})$ of the ``refinement'' $W_{\bullet}$ on a suitable divisor $H$ through the point $p$. 
(Note that the original method of \cite{AZ22}) can also be applied to divisors ``over'' $X$, 
but we only use the adjunction inequalities for divisors ``on'' $X$.)

In \cite{AZ22, AZ23} the method is mainly developed for Cartier divisors and applied to smooth Fano varieties.
In Section \ref{s:thresholds}, we explain how it extends to a singular set-up which is enough for our goals, see Set-ups \ref{setup:XLH} and \ref{setup:VWFM}. (In Subsection \ref{ss:WCI} we show that well-formed weighted complete intersections fall in this set-ups.) 
The punch line is that everything works well using divisors that are Cartier in codimension 2 and using sheaves that are Cohen--Macaulay. The final result is Lemma \ref{l:curve} where the inductive bound is given. Lemma \ref{l:curve} is based on the adjunction inequality in Theorem \ref{thm:AZadjunctionineq}.

In Section \ref{s:linearsystems}, we explain how to use hyperplane sections to reduce the dimension of $X$ to apply Lemma \ref{l:curve} and get a bound on the local delta invariant at a point not contained in the base locus of $|\cO_X(1)|$. The conclusion is Lemma \ref{l:cuttingcurve} and the condition $a_r | d$ is used here to ensure that the cutting process ends with an integral (actually, smooth) curve.

In Section \ref{s:results}, we conclude the proof of the main Theorem \ref{thm:main_intro}. The basic approach is to use Lemma \ref{l:cuttingcurve}. This cannot be done directly in the case where one considers a point contained in the base locus of $|\cO_X(1)|$. The idea is then to first carefully choose a finite cover which is unramified at the point. Proposition \ref{prop:deltafinitecover} is applied to control the delta invariant after the cover. Finally, one needs to ensure that the  bounds obtained in Theorem \ref{thm:main_intro} are good enough to prove the K-stability in Corollaries \ref{c:quasi-smooth} and \ref{c:smooth}. This numerical bounds are established in Subsection \ref{ss:numerical}.

We believe that the methods of the present paper can be used to prove K-stability of large classes of Fano weighted complete intersections. In particular, the following question may now be approachable.

\begin{question}
Let $X \subset \bP(a_0, \ldots, a_{n+1})$ be a quasi-smooth Fano hypersurface of index 1. Is $X$ K-stable?	
\end{question}	

\section{Preliminaries} \label{s:preliminaries}

We work over $\mathbb C$.
Unless otherwise stated, a variety is assumed to be reduced and irreducible. 
With a divisor on a variety $X$, we mean a Weil divisor on $X$. 

\subsection{Weighted complete intersections} \label{ss:WCI}

Given a weighted projective space $\bP=\bP(a_0, \ldots, a_{n})$,    
we set $c_1:= |\{i \in \{0,\ldots,n\} : a_i =1 \}|$ and write $\bP=\bP(1^{c_1}, a_{c_1}, \ldots, a_{n})$.

\begin{definition}\label{defn:WCIgeneralsetting}(cf. \cite{Dolgachev}, \cite{Fletcher00}, \cite{Przyjalkowski:2023aa}) 
Let $n \ge 1$ and $(a_0, \ldots , a_n) \in \bZ_{>0}^{n+1}$. 
Let $S:= \bC[z_0, \ldots ,z_n]$ such that 
$\deg z_j = a_j$ for $j=0, \ldots , n$, $I \subset S$ be a homogeneous ideal and let $R:= S/I$. 
Let 
\[
X = \Proj R \subset \bP=\bP(a_0, \ldots , a_{n}):= \Proj S 
\] be the closed subscheme of the weighted projective space $\bP$ defined by $I$. 
Assume that $X$ is irreducible and reduced (for simplicity). 
Let $ R = \bigoplus_{i=0}^{\infty} R_i$ be the decomposition into the degree $i$ parts $R_i$. 
\begin{enumerate}
\item[(0)] Let $C_X:= \Spec (S/I) \subset \bA^{n+1}$ be the affine scheme whose defining ideal is also $I \subset \bC[z_0, \ldots ,z_n]=S$ 
and $C_X^*:= C_X \setminus \{0 \}$. Then we have a natural morphism $\pi \colon C_X^* \to X$. 
\item[(i)] We say that $\bP$ is {\it well-formed} if $\gcd (a_0, \ldots , \hat{a_i}, \ldots , a_n) =1$ for all $i$. 
We say that $X \subset \bP$ is {\it well-formed} if $\bP$ is well-formed and  
$\Sing \bP \cap X \subset X$ has codimension $\ge 2$. 
\item[(ii)] We say that $X$ is {\it quasi-smooth} if $C_X^*$ is smooth. 
Let 
\[
\NQS(X) := \pi (\Sing C_X^*)
\] be the {\it non-quasi-smooth locus} of $X$. 
\item[(iii)]
Let $\cO_X(k)$ be the coherent sheaf on $X$ associated to the graded $R$-module $R(k)$ whose degree $i$ part is $R(k)_i = R_{k+i}$.   
\item[(iv)]We say that $X \subset \bP$ is a {\it weighted complete intersection (WCI) of codimension $c$} if the codimension of  $X \subset \bP$ is $c$ and 
$I$ is generated by a regular sequence $f_1,\ldots , f_c$ consisting of homogeneous polynomials. We say that $X$ is a WCI of {\it multidegree} $(d_1, \ldots , d_c)$ 
if $\deg f_j = d_j$ for $j=1, \ldots ,c$.  
\end{enumerate}
\end{definition}

\begin{remark}
Note that the definition of well-formed subvarieties given in \cite[6.9 Definition]{Fletcher00} is a bit misleading (see \cite{Przyjalkowski:2023aa}). 
\end{remark}

\begin{remark} 
Let $X \subset \bP$ be a WCI determined by $S$ and $I$ as in Definition \ref{defn:WCIgeneralsetting}. 
The multidegree $(d_1, \ldots ,d_c)$ such that $d_1 \le \cdots \le d_c$ is determined by $S, I$ and independent of choice of a regular sequence $f_1, \ldots , f_c$ as follows.  

We see that 
\[
d_1 = \min \{d \mid \dim I \cap S_d >0 \}
\]
and $d_1 = \cdots = d_{i_1}<d_{i_1 +1}$ for $i_1:= \dim I \cap S_{d_1}$. 
Let $I_1 \subset S$ be the homogeneous ideal generated by $I \cap S_{d_1}$. 
Then we see that 
\[
d_{i_1+1} = \min \{d \mid \dim I_1 \cap S_d < \dim I \cap S_d \}
\]
and $d_{i_1 + 1}= \cdots = d_{i_1 + i_2}$ for $i_2 := \dim (I \cap S_d/ I_1 \cap S_d)$. 
We can repeat this procedure and recover $(d_1, \ldots , d_c)$ from $S$ and $I$. 
\end{remark}

\begin{proposition}\label{prop:WCIproperties}
Let $S, I, R$ and $X = \Proj R \subset \bP$ be as in the setting of Definition \ref{defn:WCIgeneralsetting}
such that $\depth_{\mathfrak{m}} R \ge 2$ (e.g. $X$ is a WCI of dimension $\ge 1$), where  $\mathfrak{m}= \bigoplus_{i>0} R_i =(z_0, \ldots, z_n)$ is the irrelevant ideal defining the origin $0 \in C_X$. 
Let $i \colon X \to \bP$ be the inclusion and $k \in \bZ$.  

\begin{enumerate}
\item[(i)] We have $H^0 (\bP, \cO_{\bP}(k)) \simeq S_k$ and $H^0(X, \cO_X(k)) \simeq R_k$. 
\item[(ii)] 
Assume that $\bP$ is well-formed. 
Let $X^{\circ}:= X \setminus \Sing \bP$ and $j \colon X^{\circ} \to X$ be the inclusion.  
Then the natural homomorphism $i^* \cO_{\bP}(k) \to \cO_X(k)$ (cf. \cite[Proposition (2.8.8)]{EGAII}) induces an isomorphism 
\begin{equation}\label{eq:j_*j^*}
j_* j^* (i^* \cO_{\bP}(k)) \xrightarrow{\simeq} j_* j^* \cO_X(k). 
\end{equation}
\item[(iii)]  Assume that $X$ is a WCI (that is, $I$ is generated by a homogeneous regular sequence). 
Then the sheaf $\cO_X(k)$ is CM. Moreover, if $X$ is a well-formed WCI such that $\dim X \ge 1$, then $\cO_X(k)$ is isomorphic to the sheaves in (\ref{eq:j_*j^*}). 
\end{enumerate}
\end{proposition}

\begin{proof}
\noindent(i)
We have the following by the standard construction. 
\begin{claim}
We can check that 
\begin{equation}\label{eq:C_X^*Spec}
C^*_X \simeq \Spec_{\cO_X} \bigoplus_{m \in \bZ} \cO_X(m)=:C''_X. 
\end{equation} 
\end{claim}
\begin{proof}[Proof of Claim]
Let $U_i:= D(z_i) \subset C_X$ for $i=0, \ldots , n$ be the open affine subset as in \cite[II, Proposition 2.2]{Hartshorne77}, where we regard $z_i \in S/I$ 
corresponding to $z_i \in S$. 
Let $V_i:= \pi^{-1}(D_+(z_i)) \subset C''_X$, where $\pi \colon C''_X \to X$ is the natural morphism and 
$D_+(z_i)$ is the open affine subset as in \cite[II, Proposition 2.5]{Hartshorne77}. 
Then we have a natural isomorphism $\phi_i \colon U_i \to V_i$ for all $i$ induced by the isomorphism 
\[
\bigoplus_{m \in \bZ} \cO_X(m) \left(D_+(z_i) \right) \simeq \bigoplus_{m \in \bZ} R(m)_{(z_i)} \simeq R_{z_i},  
\]
where, for a graded $S$-module (or $R$-module) $M$ and a homogeneous element $f \in S$ (or $f \in R$), 
we let $M_{(f)}$ be the degree $0$ part of the localization $M_f$ (cf. \cite[II, Exercise 5.17]{Hartshorne77}). 
These $\phi_i$ glue and give the required isomorphism. 
\end{proof}
Since we have an exact sequence 
\[
0 \to H^0(\cO_{C_X}) \to H^0(\cO_{C^*_X}) \to H^1_{\mathfrak{m}} (\cO_{C_X})=0, 
\]
$H^0(\cO_{C_X}) \simeq R$ and $H^0(\cO_{C^*_X}) \simeq \bigoplus_{m \in \bZ} H^0(\cO_X(m))$, 
we obtain $H^0(X, \cO_X(k)) \simeq R_k$ as eigen-subspaces. 

\vspace{2mm}

\noindent(ii) 
Let $f \in S$ (resp. $\bar{f} \in R$)  be a homogeneous element and let $D_+(f) \subset \bP$ (resp. $D_+(\bar{f})$) be its associated open subset. 
Then the natural homomorphism $i^* \cO_{\bP}(k)(D_+(\bar{f})) \to \cO_X(k)(D_+(\bar{f}))$ is isomorphic to   
\[
\gamma_f \colon S(k)_{(f)} \otimes_{S_{(f)}} (S/I)_{(f)} \to ((S/I)(k))_{(f)}. 
\]

For $J \subset \{0,1,\ldots, n \}=:\ol{n}$, let $a_J:= \gcd \{a_j \mid j \in J \}$ 
and $z_J:= \prod_{j \in J} z_j$. 
For $p=[p_0, \ldots , p_n] \in \bP$, let $J_p:= \{j \in \ol{n} \mid p_j \neq 0 \}$.  
Since $\bP$ is well-formed, we have $\dps{\Sing \bP = \bigcup_{J, a_J>1} \{ z_j=0 \mid j \in \ol{n} \setminus J  \}}$ 
and that $p \in \Sing \bP$ if and only if $a_{J_p} >1$ (cf. \cite[Proposition 7]{MR0806423}). Hence we see that $p \in \bP \setminus \Sing \bP$ iff $a_{J_p}=1$, thus we have  
\[
\bP \setminus \Sing \bP = \bigcup_{J, a_J=1} D_+(z_J).
\]  

We see that the homomorphism $\gamma_{z_J}$ is an isomorphism when $a_J =1$ by 
$S(k)_{(z_J)} \simeq S_{(z_J)}$. 
Hence we see that 
$j^* (i^* \cO_{\bP}(k)) \to  j^* \cO_X(k)$ is an isomorphism, thus obtain the required isomorphism. 



\vspace{2mm}

\noindent(iii) We have $\pi_* \cO_{C_X^*} \simeq \bigoplus_{k \in \bZ} \cO_X(k)$ by (\ref{eq:C_X^*Spec}). 
It is enough to check the vanishing of the local cohomology group 
\[
H^l_{p} (\cO_X(k)) = 0
\]
for all $p \in X$ and $l \le \dim X-1$. 
This follows from 
\[
0=H^l_{\pi^{-1}(p)}(C_X^*, \cO_{C_X^*}) \simeq H^l_{p} (X, \pi_* \cO_{C_X^*}) \simeq \bigoplus_{k \in \bZ} H^l_{p}(X, \cO_X(k)).  
\]
Indeed, the vanishing follows from \cite[Theorem 3.8]{MR0224620}  since $C_X^* \subset C_X$ has only l.c.i. singularities, thus CM. 
We also obtain the first isomorphism from an isomorphism of functors $H^{0}_{\pi^{-1}(p)} \simeq H^0_{p} \circ \pi_*$ and the Grothendieck spectral sequence since $\pi$ is an affine morphism. 

Hence we see that $\cO_X(k)$ is CM. 
By the well-formedness of $X$ and \cite[Lemma 10.6]{MR4566297}, we see that $ j_* j^* \cO_X(k) \simeq \cO_X(k)$. 
\end{proof}

\begin{remark}
When $X \subset \bP$ is not well-formed, the sheaf $\cO_X(k)$ can be strange. 
For example, if $X=(z_0=0) \subset \bP(1,1,2)$, then $\cO_X(1) \simeq \cO_X$ and   
$j_* j^* \cO_X(1) \not\simeq \cO_X(1)$ for the inclusion $j \colon X \setminus \{[0,0,1] \} \to X$.  
\end{remark}

\begin{corollary}
Let $X= X_{d_1, \ldots ,d_c} \subset \bP(a_0, \ldots ,a_n)$ be a well-formed WCI of multidegree $(d_1, \ldots, d_c)$. 
Let $\iota_X:= \sum_{i=0}^n a_i -\sum_{j=1}^c d_j$ be the Fano index of $X$ and assume that $\dim X= n-c \ge 1$. 
Let $\omega_X$ be the dualizing sheaf of $X$ (which exists since $X$ is CM). 

Then we have 
\[
\omega_X \simeq \cO_X(-\iota_X) \simeq j_* j^* (i^* \cO_{\bP}(-\iota_X)), 
\] 
where $i \colon X \to \bP$ and $j \colon X \setminus \Sing \bP \to X$ are inclusions. 
\end{corollary}

\begin{proof}
Let $H_j:= (f_j=0) \subset \bP$ for $j=1, \ldots , c$ be the divisors defined by defining polynomials $f_1, \ldots , f_c$ of $X$. 
Let $\Gamma_k:= H_1 \cap \cdots \cap H_k$ for $k=1, \ldots , c$. 
Then we see that $\Gamma_k$ is pure dimensional and $\Gamma_{k} \cap \Sing \bP \subset \Gamma_k$ has codimension $\ge 2$ 
since $X \subset \bP$ is well-formed and $H_j$ is ample for all $j$. 
Ler $\Gamma_k^{\circ}:= \Gamma_k \setminus \Sing \bP$ (and $X^{\circ}= \Gamma_c^{\circ}$). 
Then $\Gamma_{k+1}^{\circ} \subset \Gamma_k^{\circ}$ is a Cartier divisor for $k=1,\ldots ,c-1$. 
Then, by the adjunction for Cartier divisors, we see that 
\[
\omega_{X^{\circ}} \simeq \omega_{\bP^{\circ}} \left( \sum_{j=1}^c H_j^{\circ} \right)|_{X^{\circ}} \simeq \cO_{\bP}(- \iota_X)|_{X^{\circ}}. 
\] 
By taking $j_*$ for the open immersion $j \colon X^{\circ} \to X$,  we see that $\omega_X \simeq \cO_X(- \iota_X)$ since $\omega_X$ is $S_2$. 
\end{proof}

\begin{remark}
Although some of the results in this subsection should hold in more general settings, 
we assumed that $X$ is irreducible and reduced since we treat such WCIs in this paper.  
\end{remark}

The following easy observation will be used in the proof of the main results in Section \ref{s:results}.

\begin{lemma}\label{lem:P_r}
	Let $X=X_d \subset \mathbb P(a_0,\ldots, a_n)=:\mathbb{P}$ be a quasi-smooth  weighted hypersurface of degree $d$, which is not a linear cone. Assume that $a_r | d$ for some $r$. Then, up to a linear automorphism of $\bP$, we can assume that $P_r \notin X$, where $P_r=[0,\ldots,0,1,0,\ldots,0]$ is the $r$-th coordinate point of $\mathbb P$.	
\end{lemma}

\begin{proof}
	Assume $P_r \in X$. Since $X$ is quasi-smooth there exists $j$ such that $\frac{\partial F}{ \partial z_j}(P_r) \ne 0$. This implies that there exists a monomial in $F$ of the form $z_jz_r^{c_r}$, i.e. $d=a_j+c_ra_r$, which tells us that $a_r |a_j$. 
	We can then consider an automorphism of the form $z_j \mapsto z_j+\lambda z_r^{a_j/a_r}$ with $\lambda \in \bC^*$ general, fixing $P_r$, to obtain  the monomial $z_r^{d_r}$ in the expression of $F$, where $d_r:=d/a_r$, which implies $P_r \notin X$. 	
\end{proof}

\subsection{Numerical observations}\label{ss:numerical}

\begin{lemma}\label{l:bound1}
	Let $X_d \subset \bP(a_0,\ldots,a_{n+1})$ be a quasi-smooth well-formed Fano weighted hypersurface of index 1 such that $a_0 \le a_1 \le \ldots \le a_{n+1}$. Assume that $a_{n+1} |d$. Then 
	$$
	\frac{(n+1)a_{n+1}}{d} \ge 1,
	$$
	and if the equality holds, then $a_0 >1$.	
\end{lemma}	
\begin{proof} 
	Write $d=ka_{n+1}$ for a positive integer $k$. If $k \ge n+2$, then $d \ge \sum_{i=0}^{n+1}a_i$, which contradicts the Fano condition, so $k \le n+1$, which is equivalent to the desired inequality.
	
	If the equality holds and $a_0=1$, then $a_1=\ldots = a_{n+1}$ because $X_d$ is Fano of index 1, which contradicts the well-formedness of $\bP$, so $a_0 >1$.
\end{proof}

The equality in Lemma \ref{l:bound1} is possible, for example considering $X_{90} \subset \bP(18^3,17,15,5)$.

In the case of smooth hypersurfaces we can give a stronger bound than the one in Lemma \ref{l:bound1}.

\begin{lemma}\label{l:ratio}
	Let $1<b_1<b_2 <\ldots < b_k$ be coprime integers with $k \ge 3$. Then
	\[
	\frac{\sum_{i=1}^k b_i}{\prod_{i=1}^k b_i} \le \frac{1}{3}.
	\]
	
\end{lemma}
\begin{proof} 
	It is easy to see that the maximum value of the expression is for $k=3$ and $b_1=2$, $b_2=3$, $b_3=5$, which gives $1/3$.
\end{proof}

Recall that $X_d \subset \mathbb P(a_0,\ldots,a_{n+1})$ is said to be a linear cone if $d=a_i$ for some $i$.

\begin{proposition}\label{p:bound}
	Let $X=X_d \subset \bP(a_0,\ldots,a_{n+1})$ be a smooth Fano weighted hypersurface of degree $d$ which is not a linear cone and it has index $\iota_X$. Assume $a_0 \le a_1 \le \ldots \le a_{n+1}$ and $a_{n+1} >1$. Set
	$$
	\gamma:= \frac{(n+1)a_{n+1}}{\iota_X d}. 
	$$
	
	\begin{enumerate}
		\item\label{0} We have
		$$
		\gamma \ge \frac{a_{n+1}}{2\iota_X},
		$$
		and the equality holds only for $X_{2(n+1)} \subset \bP(1^{(n)},2, n+1)$ when $n$ is even (which has $\iota_X=1$). 
		
		\item\label{1} If $\iota_X \le 3$ and $a_{n}=1$, then $ \gamma > 1$ except in the following cases
		
		\begin{center}
			\begin{tabular}{ |c|c|c|c| } 
				\hline
				$\iota_X$ & $d$ & $a_{n+1}$ & $\gamma$ \\  \hline
				$2$ & $d$ & $2$ & $1$ \\ 
				$3$ & $2k$ & $2$ & $2/3 + 1/(3k) < 1$ \\
				$3$ & $d$ & $3$ & $1$ \\
				\hline
			\end{tabular}.
		\end{center}

		\item\label{2} 	If $\iota_X \le 3$, $a_{n-1}=1$ and $a_n>1$, then $ \gamma > 1$ except in the following cases
		
		\begin{center}
			\begin{tabular}{ |c|c|c|c|c| } 
				\hline
				$\iota_X$ & $d$ & $a_{n}$ & $a_{n+1}$ & $\gamma$ \\  \hline
				$2$ & $6$ & $2$ & $3$ & $1$ \\ 
				$3$ & $d$ & $2$ & $3$ & $1-1/d < 1$ \\
				$3$ & $12$ & $3$ & $4$ & $1$ \\
				\hline
			\end{tabular}.
		\end{center}
		
		\item\label{3} If $\iota_X \le 3$ and $a_{n-1}>1$, then $\gamma > 1$.

	\end{enumerate}	
\end{proposition}
\begin{proof}
	Since $X$ is smooth, we know (see \cite[Corollary 3.8]{PST17} or \cite[Lemma 3.3]{PS20})  that $a_i | d$ for any $i=0,\ldots,n+1$ and $\gcd(a_i,a_j)=1$ for any $i \ne j$. Recall that 
	\begin{equation}\label{eq:adjunctiona_id}
	\sum_{i=0}^{n+1}a_i=d +	\iota_X
	\end{equation} by the adjunction formula.
	
	\smallskip
	
	\noindent{\bf (Case 1)} Assume $a_n=1$ and write $d=k a_{n+1}$ for some integer $ k \ge 2$. Then  $n+1=(k-1)a_{n+1} +\iota_X$ which gives
	\begin{equation}\label{eq:case1}
	\frac{(n+1)a_{n+1}}{d}= \frac{((k-1)a_{n+1} +\iota_X)a_{n+1}}{ k a_{n+1}} =\frac{k-1}{ k}a_{n+1} + \frac{\iota_X}{k} > \frac{a_{n+1}}{2}.
	\end{equation}
	
	This proves Item \eqref{0} in the case $a_n=1$ since, by dividing the equation (\ref{eq:case1}) by $\iota_X$, we see that  
	$\dps{\gamma = \frac{k-1}{k} \frac{a_{n+1}}{\iota_X} + \frac{1}{k} > \frac{a_{n+1}}{2\iota_X}}$. 
	
	To prove Item \eqref{1}, assume first that $\iota_X=2$, then (\ref{eq:case1}) implies 
	$$
	\frac{(n+1)a_{n+1}}{d} =\frac{(k-1)a_{n+1}+2}{ k} \ge 2
	$$
	with equality iff $a_{n+1}=2$ (For example, we see this from $(k-1)a_{n+1}+2-2k = (k-1)(a_{n+1}-2)$). 
	If $\iota_X=3$, then (\ref{eq:case1}) implies 
	$$
	\frac{(n+1)a_{n+1}}{d} =\frac{(k-1)a_{n+1}+3}{k} > 3
	$$
	unless $a_{n+1}=2$ or $a_{n+1}=3$ for which the results are in the table (we see this from $(k-1)a_{n+1} + 3 -3k = (k-1)(a_{n+1}-3)>0$ if $a_{n+1}>3$). This finishes the proof of Item \eqref{1}. 
	
	\smallskip	
	
	\noindent{\bf (Case 2)} Assume now $a_{n-1}=1$ and $a_n>1$. By (\ref{eq:adjunctiona_id}) and $\iota_X \ge 1$, we obtain 
	\[
	n+1 = d-a_n-a_{n+1}+1 + \iota_X \ge d-a_n-a_{n+1} +2
	\]
	and so
	\begin{equation}\label{eq:Case2former}
	\frac{(n+1)a_{n+1}}{d}\ge \frac{(d-a_n-a_{n+1}+2)a_{n+1}}{d}= a_{n+1} - \frac{a_n+a_{n+1}-2}{d} a_{n+1}
	\end{equation}
	with equality iff $\iota_X=1$.
	Since $1<a_n<a_{n+1}$ and $a_{n}a_{n+1}|d$ by the smoothness of $X_d$, it is easy to check that 
	\begin{equation}\label{eq:Case2latter}
	\frac{a_n+a_{n+1}-2}{d} \le \frac{1}{2}.
	\end{equation}
	Indeed, we have $(a_n+a_{n+1}-2)/d \le (a_n+a_{n+1}-2)/a_n a_{n+1}$ and 
	\[
	a_n a_{n+1} - 2 (a_n + a_{n+1} -2) = (a_n -2)(a_{n+1} -2) \ge 0 
	\]
	and the equality holds only when $a_n=2$ and $d= 2 a_{n+1}$. The inequalities (\ref{eq:Case2former}) and (\ref{eq:Case2latter}) proves Item (\ref{0}) when $a_{n-1}=1$ and $a_n >1$. When the equality holds, the conditions $\iota_X=1$, $a_n=2$ and $d= 2 a_{n+1}$ imply 
	\[
	n+2+a_{n+1} = 2a_{n+1} +1,  
	\]
	thus we have $a_{n+1} = n+1$. The smoothness of $X$ implies that $n$ is even. 
	
	Thanks to Item (\ref{0}), if $\iota_X=1$, then we obtain $\dps{\gamma \ge \frac{a_{n+1}}{2} >1}$ by $a_{n+1} > a_n >1$. If $\iota_X=2$ and $a_{n+1} >4$, then (\ref{0}) implies 
	$\dps{\gamma \ge \frac{a_{n+1}}{4}>1}$. If $\iota_X =3$ and $a_{n+1} >6$, then (\ref{0}) implies 
	$\dps{\gamma \ge \frac{a_{n+1}}{6}>1}$. 
	Hence, to prove Item \eqref{2} we only need to check by hand the cases $a_{n+1} \le 4$ for $\iota_X=2$ and $a_{n+1} \le 6$ for $\iota_X=3$. For convenience, we collect the results in the following table (recall that the degree $d$ of $X$ is divisible by $a_n a_{n+1}$). 
	
	\begin{center}
		\begin{tabular}{ |c|c|c|c|c| } 
			\hline
			$\iota_X$  & $a_{n}$ & $a_{n+1}$ & $\gamma$ \\  \hline
			$2$ &  $2$ & $3$ & $3(d-2)/(2d) > 1$ \mbox{ unless } $d=6$ \\ 
			$2$ &  $3$ & $4$ & $2(d-4)/d >1$ \\ 
			$3$ &  $2$ & $3$ & $1-1/d < 1$ \\
			$3$ &  $2$ & $5$ & $5(d-3)/(3d) > 1$ \\
			$3$ &  $3$ & $4$ & $4(d-3)/(3d) > 1$ \mbox{ unless } $d=12$ \\
			$3$ &  $3$ & $5$ & $5(d-4)/(3d) > 1$ \\
			$3$ &  $4$ & $5$ & $5(d-5)/(3d) > 1$ \\
			$3$ &  $5$ & $6$ & $6(d-7)/(3d) > 1$ \\
			\hline
		\end{tabular}
	\end{center}

	
	\smallskip 	
	
	\noindent{\bf (Case 3)} Finally, assume $a_{n-1}>1$ and recall $c_1= |\{i \in \{0,\ldots,n+1\} \colon a_i =1 \}|$ so that $n+2=c_1 + |\{i \ \colon \ a_i >1\}|$.  Then we have 
	$$
	n > c_1=d - \sum_{i >c_1} a_i+\iota_X,
	$$
	which gives
	\begin{equation}\label{eq:Case3}
	\frac{(n+1)a_{n+1}}{d} > \frac{(d - \sum_{i >c_1}a_i)a_{n+1}}{d} =a_{n+1}- \frac{a_{n+1} \cdot \sum_{i >c_1} a_i}{d} \ge \frac{2}{3}a_{n+1}, 
	\end{equation}
	where the last inequality follows from Lemma \ref{l:ratio} and $a_{n-1}>1$. This concludes the proof of Item \eqref{0} and we see that the equality can hold only in {\bf (Case 2)}.
	Item \eqref{3} follows from (\ref{eq:Case3}) noting that $a_{n+1} \ge 5$, and so $\gamma > 10/3\iota_X$. 
\end{proof}

\section{Linear systems on weighted hypersurfaces} \label{s:linearsystems}

In this section we study some properties of linear systems on weighted hypersurfaces that we need to construct a flag to apply the Abban-Zhuang method as developed in Section \ref{s:thresholds}.

\begin{lemma}\label{l:NQSnonwellformed}
	Let $X \subset \bP=\bP(a_0, \ldots, a_{n+1})$ be a weighted hypersurface such that  $1=a_0 \le a_1 \le a_2 \le \ldots  \le a_{n+1}$, $n \ge 2$ and $p=[1,0, \ldots , 0] \in X$. Let $B_1:=(z_0= \cdots = z_{c_1-1} =0) \subset X$.
	Assume that 
	\[
	\NQS(X) \subset B_1 \cup M,
	\]
	where $M \subset X$ is some finite set such that $p \notin M$.
	For $j=0,\ldots , n+1$, let $\cL(j)_p \subset \bC[z_0, \ldots , z_{n+1}]$ be the linear system of weighted homogeneous polynomials of degree $a_j$ on $\bC^{n+2}$ which vanish at $(1,0, \ldots , 0)$.  
	
	\begin{enumerate}
		\item\label{n+1''} For a general $ G_{n+1} \in \cL(n+1)_p$, the divisor $H:= (G_{n+1} =0) \subset X$ satisfies $\NQS(H) \subset B_1$.  
		\item\label{n''} Assume that $V_n:=\bigcap_{i=1}^{n} (z_i=0) \cap X$ is a finite set and $V_n \cap M = \emptyset$. 
		For a general $G_n \in \cL(n)_p$, the divisor 
		$H:= (G_n=0) \subset X $ satisfies $\NQS(H) \subset B_1$. 
	\end{enumerate}
	
\end{lemma}	

\begin{proof}
	Let $F(z_0, \ldots , z_{n+1})$ be the defining polynomial of $X$ and let $\pi \colon \bC^{n+2} \setminus \{0 \} \to \bP$ 
	be the quotient morphism. We can assume $M \cap B_1 = \emptyset$.
	
	\smallskip
	
	\noindent Proof of Item \ref{n+1''}. We first show that 
	\begin{equation}\label{eq:baselocusn}
		\pi(Bs(\cL(n+1)_p)) \subseteq B_1 \cup \{p\}.
	\end{equation}
	Let $q=[q_0, \ldots, q_{n+1}] \in X \setminus ( B_1 \cup \{p\})$.  If $q_0 \ne 0$, then $q_k \ne 0$ for some $1 \le k \le n+1$. In this case a monomial of the form $z_0^{a_{n+1}-a_k}z_k$ vanishes in $p$ but not in $q$ and so $q \notin Bs(\cL(n+1)_p)$.
	If $q_0=0$, then $q_k \ne 0$ for some $k \ge 1$ such that $a_k=1$ and the monomial $z_k^{a_{n+1}}$ vanishes in $p$, but not in $q$. This proves the inclusion \eqref{eq:baselocusn}.   
	
	By Bertini's theorem \cite[III, Corollary 10.9, Remark 10.9.2]{Hartshorne77}, (\ref{eq:baselocusn}) implies that $\NQS(H) \subset B_1 \cup \{p\}$, where $H$ is a general element of $\cL(n+1)_p$. Here we used that $M$ is a finite set disjoint from $B_1$ and so $M \cap H = \emptyset$ for $H$ general. 
	
	Next, we need to check that $H$ is quasi-smooth at $p$. Let $G(z_0,\ldots,z_{n+1})$ be a general (weighted homogeneous) polynomial of degree $a_{n+1}$ such that $H$ is given by $F=G=0$. 
	Since $X$ is quasi-smooth at $p$, there exists $k$ such that $\partial_k F(p) \ne 0$. 
	Consider the submatrix of the Jacobian matrix of $H$ given by $\partial_\ell$ and $\partial_k$ 
where $\ell \ne k$ and $\ell \neq 0$. 
The monomial $z_0^{a_{n+1}-a_\ell}z_{\ell}$ appears in $G$ with non-zero coefficient $c$ and 
we have 
\[
\partial_{\ell} (z_0^{a_{n+1}-a_\ell}z_{\ell}) = z_0^{a_{n+1} - a_{\ell}}, \ \ \partial_k (z_0^{a_{n+1}-a_\ell}z_{\ell})(p) = 0. 
\]
(Note that $k=0$ is allowed since $z_{\ell}(p) =0$.) So we have 
\[
\begin{pmatrix}
	\partial_{\ell}F & \partial_{k}F \\ 
	\partial_{\ell}G & \partial_{k}G 
\end{pmatrix}(p) = 
\begin{pmatrix}
	\partial_{\ell}F(p) & \partial_{k}F(p) \ne 0 \\ 
	c & \partial_{k}G(p) 
\end{pmatrix}, 
\]
where $c$ is the coefficient of $z_0^{a_{n+1}-a_\ell}z_{\ell}$ in $G$. 
By using these and taking the coefficient $c$ general, 
we conclude that the minor is non-zero and this proves Item \ref{n+1''}.

\smallskip	

\noindent Proof of Item \ref{n''}.
We may write   
\[
B_1 \cup V_n = B_1 \cup \{p_1,\ldots ,p_l \}
\]
with $p_1, \ldots , p_l \not\in (z_0=0)$ since $B_1 \cap V_n \subset \{(0,\ldots ,0,1) \}$.  

We start showing that 
\begin{equation}\label{eq:baselocus}
	\pi(Bs(\cL(n)_p)) \subseteq  B_1 \cup \{p_1, \ldots , p_l\}.
\end{equation}
Let $q=[q_0, \ldots, q_{n+1}] \in X \setminus ( B_1 \cup V_n)$.  If $q_0 \ne 0$, then $q_k \ne 0$ for some $1 \le k \le n $. 
In this case, a monomial of the form $z_0^{a_{n} -a_k}z_k$ vanishes in $p$ but not in $q$ and so $q \notin \pi(Bs(\cL(n)_p))$.
If $q_0=0$, then $q_k \ne 0$ for some $k \ge 1$ such that $a_k=1$ and the monomial $z_k^{a_{n}}$ vanishes in $p$, but not in $q$. This proves inclusion \eqref{eq:baselocus}. 

By Bertini's theorem, \eqref{eq:baselocus} implies that $\NQS(H) \subset B_1 \cup V_n$, where $H$ is a general element of $\cL(n)_p$. Here we used that $M$ is a finite set disjoint from $B_1$ and $V_n$ and so $M \cap H = \emptyset$ for $H$ general. 

Let $q \in \{p_1, \ldots , p_l \}$ so that $q= [q_0,0,\ldots , 0,q_{n+1}]$ (and $q_0 \neq 0$). 
The same argument as in Item \ref{n+1''} shows that a general $H$ is quasi-smooth at $q$. 
(Indeed, let $k$ be such that $\partial_k F(q) \neq 0$ which exists by the quasi-smoothness of $X$ at $q$, and 
take $\ell \in \{1,\ldots ,n \}$ such that $\ell \neq k$. Then we can follow the argument.) 	
Being $V_n$ finite, we conclude that $\NQS(H) \subset B_1$. 	
\end{proof}

\begin{remark}\label{r:automorphism}
Let $j \in \{n,n+1\}$. The conclusion of Lemma \ref{l:NQSnonwellformed} can be reformulated saying that there exists an automorphism $\phi$ of $\bP$ such that $\phi(p)=p$ and 
$$
\NQS(\phi(X) \cap \{z_j=0\}) \subset Bs|\cO_{\phi(X)}(1)|. 
$$ 

In fact, a general divisor $H$ as in the lemma is given by the zero-locus of a polynomial $G(z_0,\ldots,z_{n+1})$ of degree $a_j$ for which $G(p)=0$. Since $G$ is general, it contains the monomial $z_j$ with non-zero coefficient. This implies that we can consider the automorphism $\phi$ given by $z_i \mapsto z_i$ for any $i \neq j$ and $z_j \mapsto G$.
\end{remark}

\begin{lemma}\label{l:cuttingcurve}
Let $X=X_d \subset \bP(a_0, \ldots, a_{n+1})$ be a well-formed weighted hypersurface with $1=a_0 =a_1 \le a_2 \le \ldots  \le a_{n+1}$ and $n \ge 2$. Let $p \in X$ be a point such that $p \notin Bs|\cO_X(1)|$. 
Assume the following:  
\begin{enumerate}
	\item[(a)]	
	$\NQS(X) \subset Bs |\cO_X(1)| \cup M$ for some finite set  
	$M \subset X$ such that $p \notin M$. 
	\item[(b)] There is $r \in \{2,\ldots,n+1\}$ such that $a_r >1$ and the $r$-th coordinate point $P_r=[0,\ldots,0,1,0,\ldots,0] \notin X$. 
	\item[(c)] If $r=n+1$ and $M \neq \emptyset$, we also assume that $M \subset (z_0 =0)$ and $p \notin (z_0=0)$. 
\end{enumerate}
Let $J=\{2,\ldots, n+1\} \setminus \{r\}$. 

Then $X$ is normal. Moreover, we can take hyperplanes $H_j \in |\cO_X(a_j)|$ through $p$ for $j \in J$  such that 
\begin{enumerate} 
	\item[(i)] $\dps{\Gamma_i := \bigcap_{j \in J, i \le j } H_j}$ is normal for $i=2, \ldots , n+1$. In particular, $C= \Gamma_2= \bigcap_{j \in J} H_j$ is a smooth curve passing through $p$. 
	\item[(ii)] $Z_{H_i} \cap \Gamma_i \subset \Gamma_i$ has codimension $\ge 2$ for $i=2, \ldots , n+1$, where $Z_{H_i}:= \{x \in X \mid H_i \text{ is not Cartier at $x$} \}$. 
\end{enumerate}
\end{lemma}
\begin{proof}  
Let $F$ be the defining polynomial of $X$. Condition (a) and $P_r \notin X$ imply that 
\[
Bs |\cO_X(1)| \subset (z_0 = z_1=F=0) \subsetneq (z_0=z_1=0) \subset \bP, 
\]
we see that $\NQS(X) \subset X$ has codimension $\ge 2$. 
Since $X$ is CM, it satisfies Serre's ($S_2$) condition by Proposition \ref{prop:WCIproperties}(iii), which implies that $X$ is normal.   

Since $p \notin Bs| \cO_X(1)|$ up to reordering the coordinates with weights $1$ and applying automorphims of the form $z_0 \mapsto z_0$ and $z_i \mapsto z_i - \lambda_i z_0^{a_i}$ with $\lambda_i \in \bC$ if $i>0$, we can assume that $p=[1,0,\ldots,0]$. Note that we still have 
$
\NQS(X) \subset  Bs|\cO_X(1)| \cup M
$
for some finite set $M \subset X$ such that $p \notin M$. If $r=n+1$, we do not reorder the coordinates by the assumption $p \notin (z_0=0)$ and still have $M \subset (z_0=0)$. 
We also have $P_r = [0,\ldots,0,1, 0 ,\ldots,0] \notin X$.


\smallskip

\noindent\textbf{Step 1.} Assume that $r <n+1$. (If $r=n+1$, we start from \textbf{Step 2}.) Let $\cI_p \subset \cO_X$ be the ideal sheaf of the point $p \in X$. 
By Lemma \ref{l:NQSnonwellformed}\eqref{n+1''}, we can take a divisor $H_{n+1} = (F= G_{n+1} =0) \in |\cO_X(a_{n+1}) \otimes \cI_p|$  such that 
$\NQS(H_{n+1}) \subset Bs |\cO_X(1)|$. 
Since the singular locus $\Sing \bP$ of $\mathbb P$ is contained in $(z_0 = z_1 =0)$, 
we see that
\[
H_{n+1} \cap \Sing \bP \subset 
(z_0=z_1=F=G_{n+1}) \subsetneq (z_0=z_1 = G_{n+1}=0) \subsetneq (z_0 = z_1 =0),  
\] 
where the first $\subsetneq$ follows from $P_r \notin X$ and $G_{n+1}(P_r)=0$, and the second one from the fact that $G_{n+1}$ is general. 
By this and $(z_0=z_1=G_{n+1}=0)$ is irreducible, we see that $H_{n+1} \cap \Sing \bP$ has codimension $\ge 2$ in $H_{n+1}$, thus $H_{n+1}$ is well-formed and $Z_{H_{n+1}} \cap H_{n+1} \subset H_{n+1}$ has codimension $\ge 2$. 
By the same argument, we see that $\NQS(H_{n+1}) \subset H_{n+1}$ has codimension $\ge 2$ by $\NQS(H_{n+1}) \subset Bs |\cO_{H_{n+1}}(1)|$, thus 
$H_{n+1}$ is normal.  

Since $H_{n+1}$ is isomorphic to a well-formed hypersurface in $\bP(a_0, \ldots , a_n)$, 
we can apply Lemma \ref{l:NQSnonwellformed}\eqref{n+1''} repeatedly to take 
$H_j = (F= G_j =0) \in |\cO_X(a_j) \otimes \cI_p|$ for $j=n, \ldots , r+1$ so that 
$\Gamma_i := \bigcap_{j \in J, i \le j } H_j$ for $i=n, \ldots , r+1$ is normal with the property (ii), isomorphic to a well-formed hypersurface 
in $\bP(a_0, \ldots , a_{i-1})$ and $\NQS(\Gamma_i) \subset Bs |\cO_{\Gamma_i}(1)|$. 


\smallskip

\noindent\textbf{Step 2.} Note that $P_r \notin X$ implies that $V_{r-1}:= \bigcap_{i=1}^{r-1} (z_i=0)  \subset \Gamma_{r+1}$ is finite. Note also that 
when $r=n+1$, we have $V_n \cap M \subset \bigcap_{i=0}^n (z_i=0) = \{P_{n+1} \}$, thus $V_n \cap M = \emptyset$ by $P_{n+1} \notin X$.  
Hence we can now repeat the same procedure applying Lemma \ref{l:NQSnonwellformed}\eqref{n''} to $ \Gamma_{r+1} = Y \subset \bP(a_0, \ldots , a_{r})$ (let $\Gamma_{n+2} := X$ when $r= n+1$).  
That is, we can take $H_i \in |\cO_X(a_i) \otimes \cI_p|$ for $i=r-1, \ldots ,2$ so that $\Gamma_i$ satisfies (i) and (ii). 
In particular, for $i=2$, we get that $C= \cap_{j \in J} H_j = \Gamma_2$ is normal, thus smooth. 
\end{proof}

The following example shows that the assumption $P_r \notin X$ in Lemma \ref{l:cuttingcurve} is necessary.

\begin{example}
Consider $X=X_5 \subset \bP(1^{(4)},2)$ given by $z_0^4z_1+z_1^5+z_2^5+z_3^5+z_3z_4^2=0$ and $p=[1,0,\ldots,0] \in X$. Then $X$ is a quasi-smooth Fano of index 1. It is easy to see that for any choice of $H_1, H_2 \in |\cO_X(1) \otimes \cI_p|$ we get a reducible curve $C=H_1 \cap H_2$ with a singular point at $[*,0,0,0,*]$. Indeed, the curve $C$ is a union of some curve and the line $(z_1=z_2=z_3=0)$, since the line is contained in the base locus of  $|\cO_X(1) \otimes \cI_p|$.
(For example, $X \cap (z_3-z_2=0) \cap (z_2-z_1=0) \simeq (z_1(z_0^4+z_1^4 + z_4^2) =0) \subset \bP(1,1,2)$ .)
\end{example}

\section{delta invariants and Abban--Zhuang method} \label{s:thresholds}

Let $(X, \Delta)$ be a {\it pair} (resp. {\it projective pair}), that is, $X$ is a normal variety (resp. {\it normal projective variety}) and $\Delta$ is an effective $\bQ$-divisor such that $K_X+ \Delta$ is $\bQ$-Cartier. 
Let $E$ be a divisor over $X$, that is, a prime divisor $E \subset Y$ on some variety $Y$ with a projective birational morphism $Y \to X$. 
For a divisor $E$ over a variety $X$, let $A_{(X,\Delta)} (E)$ be the log discrepancy of the divisor $E$. 
For a valuation $\nu$ on a variety $X$, let $A_{(X,\Delta)} (\nu)$ be the log discrepancy of the valuation $\nu$ (cf. \cite{MR3060755}). 

Let $(X, \Delta)$ be a pair and $Z \subset X$ a closed subvariety. 
We say that a pair $(X, \Delta)$ is {\it klt near $Z$} (resp. {\it lc near $Z$}) if 
there exists an open neighborhood $U$ of the generic point of $Z$ such that $(U, \Delta|_U)$ is klt (resp. lc).

We shall give the definition of the delta invariants by using valuations in the following. 

\begin{definition}\label{defn:valuativeSdeltadeltaZ} \cite[Definitions 2.2, 2.4 and 2.5]{AZ22}
	Let $(X, \Delta)$ be a pair and $L$ be a $\bQ$-Cartier big divisor. 
	Let $Z \subset X$ be a closed subvariety.  
	\begin{enumerate}
		\item[(i)] For a valuation $\nu$ over $X$ of ``linear growth'' (e.g. divisorial, finite discrepancy, etc \cite[Definition 2.2]{AZ22}), we let 
		\[
		S(L ; \nu):= \frac{1}{\vol(L)} \int_{0}^{\infty} \vol (L; \nu \ge t) dt, 
		\] 
		where 
		\[
		\vol (L; \nu \ge t):= \lim_{m \to \infty} \frac{\dim \{ s\in H^0(X, mL) \mid \nu(s) \ge mt \}}{m^{\dim X}/ (\dim X)!}. 
		\]
		When $\nu = \ord_E$ is a valuation induced by a divisor $E$ over $X$, we write $S(L; E):= S(L; \ord_E)$. 
		When $L$ is ample and $\sigma \colon \tilde{X} \to X$ is a projective birational morphism with a prime divisor $E \subset \tilde{X}$ over $X$, we have  
		\[
		S(L; E) = \frac{\int_0^{\infty} \vol(\sigma^*L- xE) dx }{L^n}. 
		\]
		\item[(ii)] Assume that $(X, \Delta)$ is klt. We define 
		\begin{equation}\label{eq:defndelta}
		\delta(L):= \delta(X, \Delta;L):= \inf_E \frac{A_{(X,\Delta)}(E)}{S(L; E)} = \inf_{\nu} \frac{A_{(X,\Delta)}(\nu)}{S(L; \nu)}, 
		\end{equation}
		where $E$ runs over all divisors over $X$ and $\nu$ runs over all valuations over $X$. 
		It is known that the latter infimum is indeed a minimum (that is, achieved by some valuation $\nu$) by 
		\cite[Theorem E]{MR4067358}.
		\item[(iii)] Assume that $(X, \Delta)$ is klt at the generic point of $Z \subset X$.  
		We define 
		\begin{equation}\label{eq:defndeltaZ}
		\delta_Z(L):= \inf_{E, Z \subset c_X(E)} \frac{A_{(X,\Delta)}(E)}{S(L; E)} = \inf_{\nu, Z \subset c_X(\nu)} \frac{A_{(X,\Delta)}(\nu)}{S(L; \nu)}, 
		\end{equation}
		where $E$ runs over all divisors over $X$ whose center contains $Z$ and $\nu$ runs over all valuations over $X$ whose center contains $Z$. 
		The latter infimum here can also be shown to be a minimum by the same proof of \cite[Theorem E]{MR4067358} (cf. \cite[Definition 2.5]{AZ22}).  
	\end{enumerate}
\end{definition}

We can also define the delta invariants by using the notion of basis type divisors as follows, for the equivalence see \cite{MR4067358}, \cite[Definition 2.5]{AZ22}.  

\begin{definition}
	Let $(X, \Delta)$ be a pair. 
	Let $L$ be a big $\bQ$-Cartier ($\bQ$-)divisor on $X$. 
	Let $V \subset H^0(X, mL)$ be a non-zero linear subspace. 
	\begin{enumerate}
		\item[(i)] We say that a ($\bQ$-)divisor $D$ is a {\it basis type divisor} of $V$ if $V \neq \{0 \} $ and 
		$D= \sum_{i=1}^N (s_i=0)$ for some basis $s_1, \ldots , s_n$ of $V$, 
		where $(s_i=0)$ is the effective ($\bQ$-)divisor determined by $s_i \in V$. 
		(For $V=\{0 \} \subset H^0(X, mL)$, we say that $D=0$ is a basis type divisor of $V$ by convention.  ) 
		
		Let $M(L):= \left\{ m \in \bZ_{>0} \mid h^0(mL) \neq 0 \right\}$. 
		For $m \in M(L)$, we say that a $\bQ$-divisor $D$ is an {\it $m$-basis type $\bQ$-divisor} if 
		\[
		D= \frac{1}{mh^0(X, mL)} D_m  
		\]
		for some basis type divisor $D_m$ of $H^0(X, mL)$. 
		
		\item[(ii)]Let $m \in M(L)$ and $\nu \in \Val_X^*$ be a valuation of linear growth. 
		Let 
		\[
		S_m(L;\nu):= \sup_{D \sim_{\bQ} L} \nu(D), 
		\]
		where $D$ runs over all $m$-basis type $\bQ$-divisors of $L$. 
		Then we have 
		\[
		S(L;\nu) = \lim_{m \to \infty} S_m(L; \nu), 
		\] where $S(L;\nu)$ is the invariant defined in Definition \ref{defn:valuativeSdeltadeltaZ}(i). 
		\item[(iii)] 
		Let $(X, \Delta)$ be a klt pair. For $m \in M(L)$, let 
		\[
		\delta_m(L):= \sup \left\{ \lambda \ge 0 \mid (X, \Delta+ \lambda D): \text{lc ($\forall D$: $m$-basis type $\bQ$-divisor of $L$)}  \right\}. 
		\]
		Then we have $\delta(L) = \lim_{m \to \infty} \delta_m(L)$ for the invariant $\delta(L)$ defined in Definition \ref{defn:valuativeSdeltadeltaZ}(ii). 
		
		\item[(iv)] Let $(X, \Delta)$ be a pair which is klt near a closed subvariety $Z \subset X$. Let
		\[
		\delta_{Z,m}(L):= \sup \left\{ \lambda \ge 0 \mid (X, \Delta+ \lambda D): \text{lc near $Z$ ($\forall D$: $m$-basis type $\bQ$-divisor of $L$)}  \right\}. 
		\]
		Then we have $\delta_Z(L) = \lim_{m \to \infty} \delta_{Z,m}(L)$ for the invariant $\delta_Z(L)$ defined in Definition \ref{defn:valuativeSdeltadeltaZ}(iii). 
	\end{enumerate}
\end{definition}

The important property is the following relation with the K-stability. 

\begin{theorem}\label{thm:delta_stability} (\cite{MR4067358}, \cite{MR3956698}, \cite{MR3896135}, \cite{MR3715806}, \cite{LXZ22})
	Let $X$ be a klt Fano variety. Then it is K-semistable (resp. K-stable) if and only if $\delta(-K_X) \ge 1$ (resp. $\delta(-K_X) >1$) 
	if and only if $\delta_p(-K_X) \ge 1$ (resp. $\delta_p(-K_X) >1$) for all $p \in X$. 
\end{theorem}

The following proposition on the delta invariants via finite surjective morphisms is a local variant of \cite[Theorem 1.1]{Dervan}, \cite[Corollary 1.7]{MR3960299},\cite[Corollary 4.13]{MR4309493}, \cite[Theorem 1.2]{MR4380122},  and it is known to experts. We include a proof for completeness. 

\begin{proposition}\label{prop:deltafinitecover}
	Let $(X, \Delta)$ and $(Y, \Delta_Y)$ be a pair. 
	Let $\pi \colon Y \to X$ be a finite surjective morphism with 
	\[
	K_Y + \Delta_Y = \pi^* (K_X + \Delta). 
	\]   
	Let $p \in X$ and $q \in Y$ such that $\pi(q) = p$ and $(X, \Delta)$ is klt at $p$. 
	Let $L$ be an ample line bundle on $X$. 
	Then we have 
	\[
	\delta_q(Y, \Delta_Y; \pi^*L) \le \delta_p(X, \Delta; L). 
	\]
	 In addition, assume that the equality holds and that every valuation that computes $\delta_q(Y, \Delta_Y; \pi^*L)$ is induced by a prime divisor on $Y$. Then, if  $\delta_p(X, \Delta; L)$ is computed by a divisor over $X$, then there exists a prime divisor on $X$ that computes it.
\end{proposition}

\begin{proof}
	Let $E$ be a divisor over $X$ whose center contains $p$ 
	and $\sigma \colon \tilde{X} \to X$ be the associated prime blow-up (cf. \cite{MR2063783}, \cite{MR3960299}) with the exceptional divisor $E \subset \tilde{X}$ such that 
	$-E$ is $\sigma$-ample $\bQ$-Cartier divisor (thus $p \in \sigma(E)$). Note that $\sigma^{-1}(\sigma(E)) = E$ since $\sigma$ is the blow-up along some ideal $\cI \subset \cO_X$ such that $\cO_X / \cI$ is supported on $\sigma(E)$ (cf. \cite[Proposition 2.4]{MR2063783}). 
	Then we have the commutative diagram 
	\[
	\xymatrix{
		\tilde{Y} \ar[r]^{\tilde{\pi}} \ar[d]^{\sigma'} & \tilde{X} \ar[d]^{\sigma} \\
		Y \ar[r]^{\pi} & X, 
	}
	\]
	where $\tilde{Y}$ is the normalization of the fiber product. 
	Let 
	\[
	\tilde{\pi}^* (E) := \sum r_j E'_j.
	\]
	Note that 
	\[
	\bigcup E'_j = \tilde{\pi}^{-1}(E) = \tilde{\pi}^{-1} (\sigma^{-1} (\sigma(E))) = (\sigma')^{-1} (\pi^{-1} (\sigma(E)))
	\]
	and some irreducible component $Z \subset \pi^{-1}(\sigma(E))$ contains $q$. 
	Then there exists $E'_i$ which dominates $Z$ and $q \in \sigma' (E'_i)$.  
	
	The rest is the same computation as in the proof of \cite[Corollary 1.7]{MR3960299}. Since we have $A_{(Y, \Delta_Y)}(E'_i) = r_i A_{(X, \Delta)}(E)$, for $d:= \deg \pi$, we obtain 
	\begin{multline*}
	\frac{A_{(Y,\Delta_Y)}(E'_i)}{S(\pi^*L; E'_i)} = \frac{A_{(Y,\Delta_Y)}(E'_i) \cdot (\pi^*(L)^n)}{\int_0^{\infty} \vol(\sigma'^{*} \pi^*(L) - y E'_i) dy} = 
	\frac{A_{(X,\Delta)}(E) \cdot (d L^n)}{\int_0^{\infty} \vol(\sigma'^{*} \pi^*(L) - xr_i E'_i) dx} \\ 
	\le \frac{A_{(X,\Delta)}(E) \cdot (d L^n)}{\int_0^{\infty} \vol(  \tilde{\pi}^*(\sigma^*L - xE)) dx} \le  \frac{A_{(X,\Delta)}(E) \cdot ( L^n)}{\int_0^{\infty} \vol( \sigma^*L - xE) dx}  = \frac{A_{(X,\Delta)}(E)}{S(L; E)}, 
	\end{multline*}
	where we used $\vol(  \tilde{\pi}^*(\sigma^*L - xE)) \le \vol(  \tilde{\pi}^*(\sigma^*L ) - x r_i E'_i)$ in the first inequality and used 
	$
	\vol(  \tilde{\pi}^*(\sigma^*L - xE)) \ge d \vol( \sigma^*L - xE)
	$
	(cf. \cite[Lemma 4.1]{MR3960299}) in the 2nd inequality. 
	This implies the required inequality.

To prove the last statement, consider a divisor $E$ over $X$ that computes $\delta_p(X, \Delta; L)$. Then the above proof gives a divisor $E'_i$ on $\tilde Y$ such that
$$
\delta_q(Y, \Delta_Y; \pi^*L)=\frac{A_{(Y,\Delta_Y)}(E'_i)}{S(\pi^*L; E'_i)}=\frac{A_{(X,\Delta)}(E)}{S(L; E)}=\delta_p(X, \Delta; L),
$$
since equalities must hold in the chain of inequalities.
By the assumption, any valuation computing $\delta_q(Y, \Delta_Y; \pi^*L)$ is induced by a prime divisor on $Y$, which means that actually $E_i'$ lives already on $Y$ and so $E \subset X$ since $\tilde{\pi}$ is finite. 
\end{proof}

\subsection{Abban--Zhuang method}
In this subsection we explain how to adapt the Abban--Zhuang method in \cite{AZ22} to the situation we are interested in. The main difference is that we need to work with divisors that are not Cartier (or plt type).

More precisely, we are going to use the following set-up.

\begin{setup}\label{setup:XLH}
	Let $(X, \Delta)$ be a projective pair such that $\lfloor \Delta \rfloor =0$ and $L$ be an ample $\bQ$-Cartier divisor.  Assume that $\cO_X(kL)$ is CM for all $k \in \bZ$.  
	Let 
	\[
	Z_L:= \{x \in X \mid L \text{ is not Cartier at $x$} \} \subset X
	\]
	and assume that $Z_L \subset X$ has codimension $\ge 3$. 
	Let $X':= X \setminus Z_L$.  
	Let $H \in |L|$ be a reduced and irreducible divisor such that $H$ is normal and 
	$H \not\subset \Supp \Delta$. 
	Let $\Delta_H:= \Delta|_H$ be an effective $\bQ$-divisor on $H$ which satisfies  
	\[
	(K_X+ \Delta+H)|_H = K_H + \Delta_H
	\]
	(cf. \cite[Proposition 4.5(4)]{MR3057950}). 
\end{setup}

\begin{remark}
	In certain situations, it would be useful to have a local version of set-up \ref{setup:XLH}. Fix a point $x \in X$. The local set-up differs from the global one just in the facts that we assume $X$ to be $\bQ$-Gorestein only at $x$ and $H$ to be normal only at $x$. 
\end{remark}


\begin{lemma}\label{lem:restrictexact}  
	Let $X$ be a normal variety and $L$ be a $\bQ$-Cartier divisor such that $\cO_X(kL)$ is CM for any $k \in \bZ$. 
	Let $H \in |L|$ be an irreducible and reduced divisor such that there exists a closed subset $Z_L \subset X$ with $\codim_{X} Z_L \ge 3$ 
	such that $L$ is Cartier on $X':=X \setminus Z_L$. 
	Let $\iota \colon H':=H\setminus Z_L \hookrightarrow H$ be the open immersion. 
	
	Then for any $k\in \bZ$,
	the sequence 
	\[
	0 \to \cO_X((k-1)L) \to \cO_X(kL) \to \iota_*  \left( \cO_{X} (kL)|_{H'} \right) \to 0
	\]
	is exact and $\iota_*  \left( \cO_{X} (kL)|_{H'} \right)$ is CM.	
\end{lemma}

\begin{remark}\label{i:restriction} 
	If $H$ is normal, then there exists a unique $\bQ$-Cartier Weil divisor class $L|_H$ on $H$  
	such that
	\[
	\cO_H\left( k  (L|_H) \right) \cong \iota_*  \left( \cO_{X} (kL)|_{H'} \right).
	\]
\end{remark}

\begin{proof} (cf. \cite[Theorem 2.74]{MR4566297}) 
	Since the statement is local, we may assume that $X$ is affine. 
	Since $L$ is Cartier on $X'$, we have an exact sequence 
	\[
	0 \to \cO_{X'}((k-1)L|_{X'}) \to \cO_{X'}(kL|_{X'}) \to \cO_{H'}(kL|_{H'}) \to 0   
	\] 
	and $\cO_{H'}(kL|_{H'}) \cong \cO_{X'}(kL)|_{H'}$. We obtain the required exact sequence by taking $\iota_*$ of the above sequence since $R^1 \iota_* \cO_{X'}((k-1)L|_{X'}) = H^2_Z(X,(k-1)L )= 0$ 
	from the CM-condition (cf. \cite[Theorem 3.8]{MR0224620}, \cite[(10.29.5)]{MR4566297}), where $\iota \colon X' \hookrightarrow X$ denotes the inclusion. 
	
	We also see that $\iota_*  \left( \cO_{X} (kL)|_{H\setminus Z_L} \right)$ is CM by  \cite[Lemma 10.28]{MR4566297}.  
	%
\end{proof}

In Setup \ref{setup:XLH}, we also see that $(H, \Delta_H)$ is a projective pair (i.e. $K_H+ \Delta_H$ is $\bQ$-Cartier) 
by the construction. 

Let $Z \subset H$ be a closed subvariety. 
If $(H, \Delta_H)$ is klt near $Z$, then $(X, \Delta)$ is also klt near $Z$ (cf. \cite[Theorem 5.50(1)]{KM98}). 
Hence we have well-defined delta invariants $\delta_Z(X, \Delta; L)$ and $\delta_Z(H, \Delta_H; L|_H)$ which we want to compare.  
  The same definitions introduced in \cite[Section 2]{AZ22} can be given in our set-up:

\begin{definition}\label{setup:VWFM}
	Let $(X, \Delta), L, (H, \Delta_H)$ be as in Setup \ref{setup:XLH}. 
	\begin{enumerate}
		\item[(i)] Let $V_{\bullet}$ be the complete linear series associated to $L$, that is, $V_m:= H^0(X, mL)$. 
		Let $W_{\bullet, \bullet}$ be its refinement by $H$, that is, 
		\[
		W_{m,j}:= \Image \left( H^0(X, mL - jH) \to H^0(H, (mL-jH)|_H) \right). 
		\] 
		Then $W_{\bullet, \bullet}$ is an $\bN^2$-graded linear series associated to $L_1:= L|_H$ and $L_2 := -L|_H$ (see Lemma \ref{lem:restrictexact}).
		\item[(ii)] For $(m,j) \in \bN^2$ such that $W_{m,j} \neq 0$, let 
		\[
		W_{m,j} = M_{m,j} + F_{m,j}
		\]
		be the decomposition to the movable part $M_{m,j}$ and the fixed part $F_{m,j}$. 
		Let 
		\[
		F_m:= F_m(W_{\bullet, \bullet}):= \frac{1}{m h^0(W_{m,\bullet})} \sum_{j=0}^m h^0(W_{m,j}) F_{m,j}, 
		\]
		where $h^0(W_{m,j}) := \dim_{\bC} W_{m,j}$ and $h^0(W_{m,\bullet}):= \sum_{j=0}^m h^0(W_{m,j})$. 
		Let $F(W_{\bullet, \bullet}) := \lim_{m \to \infty} F_m$. 
		\item[(iii)](\cite[Definition 2.11]{AZ22}) Let 
		\[
		M(W_{\bul, \bul}):= \{m \in \bZ_{>0} \mid W_{m,j} \neq 0 \text{ for some $0 \le j \le m$} \}, 
		\]  
		and $m \in M(W_{\bul, \bul})$. 
		We say that a $\bQ$-divisor $D$ is {\it a basis type divisor of $W_{\bul, \bul}$} (resp. {\it an $m$-basis type $\bQ$-divisor of $W_{\bul, \bul}$ }) if 
		\[
		D = \sum_{j=0}^m D_j  \ \ \left( \text{resp. } D= \frac{1}{m h^0(W_{m,\bul})} \sum_{j=0}^m D_j \right) 
		\]
		for some basis type divisors $D_j$ of $W_{m,j}$ for $j=0, \ldots ,m$. 
		
		\item[(iv)] \cite[Definition 2.12]{AZ22} For $m \in M(W_{\bul, \bul})$ and a divisor $E$ over $H$, let 
		\[
		S_m(W_{\bul, \bul}; E):= \sup_D \ord_E (D),  
		\]
		where $D$ runs over all $m$-basis type $\bQ$-divisors $D$ of $W_{\bul, \bul}$. 
		Let $S(W_{\bul, \bul};E):= \lim_{m \to \infty} S_m(W_{\bul, \bul}; E)$. 
		(We can also define $S(W_{\bul, \bul};E)$ by using valuations as in \cite[Corollary 2.13]{AZ22}.)
		\item[(v)] Let $Z \subset H$ be a closed subvariety. 
		Assume that $(H, \Delta_H)$ is klt (resp. $(H, \Delta_H)$ is klt near $Z$).  
		Then we define $\delta(W_{\bullet, \bullet})$ (resp. $\delta_Z(W_{\bullet, \bullet})$) as in \cite[Lemma 2.9]{AZ22} by 
		\begin{equation}\label{eq:defndeltaWbulbul}
		\delta(W_{\bul,\bul}):= \inf_E \frac{A_{(H,\Delta_H)} (E)}{S(W_{\bul,\bul}; E)}, \ \  \left( \text{resp. } \delta_Z(W_{\bul,\bul}):= \inf_{E, x \in C_H(E)} \frac{A_{(H,\Delta_H)} (E)}{S(W_{\bul,\bul}; E)} \right),  
		\end{equation}
		where $E$ runs over divisors over $H$ (resp. $E$ runs over divisors over $H$ whose center $C_H(E)$ on $H$ contain $Z$).  
	\end{enumerate}
\end{definition}

The above $\bN^2$-graded linear series $W_{\bul, \bul}$ satisfies the following. 

\begin{lemma}\label{lem:W} 
	\begin{enumerate}
		\item[(i)]The graded linear series $W_{\bullet, \bullet}$ contains an ample linear series (\cite[Definition 2.9]{AZ22}). 
		\item[(ii)] We also have $F(W_{\bullet, \bullet})=0$ and $W_{\bullet, \bullet}$ is almost complete (\cite[Definition 2.16]{AZ22}). 
	\end{enumerate}
\end{lemma} 

\begin{proof}
	We can take $N \gg 0$ so that $H^1(X, \cO_X(k-1)L)) =0$ for all $k \ge N$ and $\cO_H(kH|_H)$ is globally generated for $k \ge N$. 
	Then we see that 
	\[
	W_{m,j} = H^0(H, \cO_H((m-j)H|_H))
	\] when $m-j \ge N$.  
	
	%
	
	\noindent(i) By the above, we can check that $W_{\bullet, \bullet}$ contains an ample series. 
	Indeed, we see that $(2N,N) \in \Inte(\Supp W_{\bullet, \bullet}) \subset \bR^2$. 
	We also see that, for $(m,j) \in \Inte (\Supp W_{\bullet, \bullet}) \cap \bN^2$, $W_{(km,kj)} \neq 0$ for $k \gg 0$.  
	For $\bfa_0:= (2N, N) \in \Inte (\Supp W_{\bullet, \bullet}) \cap \bN^2$, we have  $\bfa_0 \cdot (L_1, L_2) = NL|_H$ and $W_{m \bfa_0} = H^0(H, mNL|_H)$. 
	
	\vspace{2mm}
	
	\noindent(ii) We see that $F(W_{\bullet, \bullet}) =0$ since $F_{m,j} =0$ for $m-j \ge N$. 
	
	To show that $W_{\bullet, \bullet}$ is almost complete, we need to verify the condition in \cite[Definition 2.16.2]{AZ22}. We can take $L|_H$ as $L$ in \cite[Definition 2.16.2]{AZ22} and  since $M_{m,j} = |(m-j)L|_H|$ holds for $m-j \ge N$, the conclusion follows. (cf. \cite[Example 2.12]{AZ22}). 
\end{proof}



\begin{remark}\label{rmk:basistypedivadjunction}
Let $X, L, H$ be as in Setup \ref{setup:XLH}. Let $V_{\bul}$ be the graded linear series induced by $V_m= H^0(X, mL)$ and set $N_m:= h^0(X, mL)$.
Let $\cF_H$ be the filtration induced by $H$,  that is $\cF_H^{\lambda} V_m = \{s \in V_m \mid \ord_H(s) \ge \lambda \}$.
We say that a basis $s_1, \ldots , s_{N_m}$ of $V_m$ is compatible with $\cF_H$ if every $\cF_H^\lambda$ is the span of some $s_i$.
Let $D$ be an $m$-basis type $\bQ$-divisor of $V_{\bul}$ compatible with the filtration $\cF_H$ and 
let $s_1, \ldots , s_{N_m} \in V_m$ be the basis which is compatible with $\cF_H$ and defines $D$ (see \cite[Definition 2.11]{AZ22}).   

Then we have  
\[
D= \frac{1}{mN_m} \sum_{j=0}^m D'_j \text{, where $D'_j := \sum_{\ord_H(s_i) = j} (s_i=0)$}. 
\]
The compatibility implies that 
\[
\left| \{i \mid \ord_H(s_i) =j \} \right| = h^0(X, mL-jH) - h^0(X, mL-(j+1)H). 
\]
 By this and the definition of $S_m(V_{\bul}; H)$ (cf. \cite[Proposition 3.1]{AZ22}), we have $D = S_m(V_{\bul}; H) H + \Gamma$ for 
an effective $\bQ$-divisor $\Gamma$ of the form 
\[
\Gamma = \frac{1}{mN_m} \sum_{j=0}^m D_j \text{, where  $D_j = \sum_{i, \ord_H(s_i) = j} \left( (s_i=0) - jH \right)$}. 
\] 
Note that $D_j|_H$ is a basis type divisor of $W_{m,j}$, thus $\Gamma|_H$ is an $m$-basis type $\bQ$-divisor of $W_{\bul, \bul}$. 
\end{remark}

We have the following inequality on adjunction for delta invariants (see \cite[Theorem 3.2 and Lemma 4.3]{AZ22}).

\begin{theorem}\label{thm:AZadjunctionineq} 
	Let $(X, \Delta), L, (H, \Delta_H)$ and $W_{\bul, \bul}$ be as in Setup \ref{setup:XLH} and Definition \ref{setup:VWFM}. Let $Z \subset H$ be a closed subvariety. 
	Assume that $(H, \Delta_H)$ 
	is klt near $Z$. 
	Then we have 
	\begin{equation}\label{eq:AZineq3.2}
	\delta_Z(L) \ge \min \left\{ \frac{1}{S(L; H)}, \delta_Z( W_{\bul, \bul})  \right\} =: \lambda. 
	\end{equation}
	As a consequence, we obtain 
	\begin{equation}\label{eq:AZineq4.3}
	\delta_Z(L) \ge \min \left\{n+1, \frac{n+1}{n} \delta_Z(L|_H)  \right\}. 
	\end{equation}
	
If equality holds in Equation \eqref{eq:AZineq4.3}, then either $\delta_Z(L)=n+1$ and it is computed by $H$ or $\delta_Z(L)=\frac{n+1}{n} \delta_Z(L|_H)$ and $C_X(v) \not \subset H$ for any valuation $v$ that computes $\delta_Z(L)$. Moreover, in the latter case, for every irreducible component $S$ of $C_X(v) \cap H$ containing $Z$, there exists a valuation $v_0$ on $H$ with centre $S$ computing $\delta_Z(L|_H)$.
\end{theorem}

\begin{proof} 	
	The proof of \cite[Theorem 3.2]{AZ22} works in our set-up to show inequality (\ref{eq:AZineq3.2}). 
	In \cite[Theorem 3.2]{AZ22}, the divisor $F$ used to refine $V_\bul$ and obtain $W_{\bul, \bul}$ is assumed to be either Cartier or of plt type. 
	In our case the restriction of $L$ to $H$ behaves well thanks to Lemma \ref{lem:restrictexact} and Remark \ref{i:restriction}. 
	Also note that the adjunction formula works without a different under our assumptions. 
	It is enough to show the following claim. 

\begin{claim}\label{claim:delta_minequality}
For $m \in \bZ_{>0}$, we have 
\[
\delta_{Z,m} (L) \ge \min \left\{ \frac{1}{S_m(L; H)}, \delta_{Z,m} (W_{\bul, \bul}) \right\}=:\lambda_m.  
\]
\end{claim}

\begin{proof}[Proof of Claim]
Let $D$ be an $m$-basis type $\bQ$-divisor of $V_{\bul}$ compatible with $\cF_H$ as in Remark \ref{rmk:basistypedivadjunction} 
with the decomposition 
$
D= S_m(V_{\bul}; H) H + \Gamma 
$
such that $H \not\subset \Supp \Gamma$ and $\Gamma|_H$ is an $m$-basis type divisor of $W_{\bul, \bul}$.  
Then we have 
\begin{equation}\label{eq:lambda_mD}
K_X+ \lambda_m D = K_X + \lambda_m S_m(L; H) H + \lambda_m \Gamma,  
\end{equation}
where $\lambda_m S_m(V_{\bul}; H) \le 1$ by the definition of $\lambda_m$. 
We know that $(H, \lambda_m (\Gamma|_H))$ is lc near $Z$ by the definition of $\lambda_m$ (or $\delta_{Z,m}(W_{\bul, \bul})$). 
By the inversion of adjunction, we see that $(X, H+ \lambda_m \Gamma)$ is also lc near $Z$. 
Then, by (\ref{eq:lambda_mD}), we see that $(X, \lambda_m D)$ is lc near $Z$, thus we obtain the claim. 
\end{proof}

By letting $m \to \infty$, we obtain (\ref{eq:AZineq3.2}) from Claim \ref{claim:delta_minequality}. 
By the proof of Claim \ref{claim:delta_minequality}, we also obtain  
\begin{equation}\label{eq:A_Xnuineq}
A_X(\nu) \ge \lambda S(L; \nu) + (1- \lambda S(L;H)) \nu(H)  
\end{equation}
for all valuations $\nu$ on $X$ whose center contain $Z$. 
	
\smallskip 
	
	The proof of (\ref{eq:AZineq4.3}) is parallel to that of \cite[Lemma 4.3]{AZ22}, we explain the main steps to deduce (\ref{eq:AZineq4.3}) from (\ref{eq:AZineq3.2}) in our set-up.
	
	By definition 
	\begin{equation}
	S(L;H) = \frac{1}{\vol(L)} \int_0^{\infty} \vol (L; \nu_H \ge t) dt = \frac{1}{n+1}.
	\end{equation}
	Then the 1st term in the R.H.S. of (\ref{eq:AZineq3.2}) is $\dps\frac{1}{S(L;H)} = n+1$. 
	Hence (\ref{eq:AZineq4.3}) is reduced to show the following claim. 
	
	\begin{claim}
		We have $\dps{\delta_Z(W_{\bul, \bul}) = \frac{n+1}{n} \delta_Z(L|_H)}$. 
	\end{claim}
	
	\begin{proof}[Proof of Claim]
		The proof makes use of the first Chern class $c_1(W_{\bul, \bul})$ which is defined in  \cite[Lemma-Definition 2.14]{AZ22} .
		As explained in the beginning of the subsection \cite[3.1]{AZ22}, we have
		\[
		c_1(W_{\bullet, \bullet}) = \left( c_1(V_{\bullet}) - S(V_{\bullet}; H) H \right)|_H \sim_{\bQ} \frac{n}{n+1} L|_H
		\]
		since $c_1(V_{\bullet}) = L$ and $S(V_{\bullet}; H) = S(L;H) =  \dps{\frac{1}{n+1}}$. 
		By Lemma \ref{lem:W}, we can apply \cite[Lemmas 2.11, 2.13]{AZ22} 
		to get
		\[
		S(W_{\bullet, \bullet};E ) = \frac{n}{n+1}S(L|_H; E)
		\] 
		for divisors $E$ over $H$ whose center contains $Z$. (Indeed, we can consider an $\bN^2$-graded linear series 
		$W'_{\bul, \bul}$ defined by $W'_{m,j} = W_{km, kj}$ for $k \in \bZ_{>0}$ such that $kL$ is Cartier and can reduce to the situations where 
		we can apply \cite[Lemma 2.13]{AZ22} stated for big line bundles.)
		
		By this, we obtain the required equality by comparing the definitions (\ref{eq:defndeltaZ}) and (\ref{eq:defndeltaWbulbul}) of the delta invariants. 
	\end{proof}
	
The statement about equality follows as in \cite[Theorem 3.2]{AZ22} based on the inequality (\ref{eq:A_Xnuineq}) and \cite[Theorem A]{MR3060755}.  

\end{proof}


\begin{lemma}\label{l:curve} (cf. \cite[Lemma 4.4]{AZ22}) 
	Let $X$ be a normal $\bQ$-Gorenstein  projective variety and $L$ be an ample $\bQ$-Cartier divisor such that $\cO_X(kL)$ is CM for $k \in \bZ$. 
	Let $n:= \dim X$ and $x \in X$ be a smooth point.  
	Assume that there exist $m_i \in \bZ_{>0}$  and $H_i \in |m_i L|$ containing $x$ for $i=1, \ldots, n-1$ with the following:  
	\begin{enumerate}\label{eq:condition*}
		\item[(i)] $\Gamma_j:= H_1 \cap \cdots \cap H_j$ is normal 
		for $j=1, \ldots , n-1$.  
		In particular, $\Gamma_{n-1}= \bigcap_{i=1}^{n-1} H_i$ is a smooth curve.
		\item[(ii)] $Z_{H_{j}} \cap \Gamma_{j-1} \subset \Gamma_{j-1}$ has codimension $ \ge 2$ for $j=1,\ldots, n-1$, where $\Gamma_0:=X$.  
	\end{enumerate} 
	Then we have 
	$$
	\delta_x(L) \ge \frac{n+1}{m_1 \cdots m_{n-1}L^n}.
	$$
	If equality holds, then either $\delta_x(L)=n+1$ or every valuation that computes $\delta_x(L)$ is divisorial and induced by a prime divisor $E$ on $X$.
\end{lemma}

\begin{proof}
	We follow the proof of \cite[Lemma 4.4]{AZ22}. We show the statement by induction on $n$. 
	If $n=1$, then the proof is same as \cite[Lemma 4.4]{AZ22}.
	
	Assume that the statement has been proved in dimension $\le n-1$. 
	Since $\Gamma_j \subset \Gamma_{j-1}$ is a Cartier divisor at $x$ for $j \ge 1$ and $\Gamma_{n-1}$ is smooth at $x$,  
	we see that $\Gamma_j$ is smooth at $x$ for $j=n-2, \ldots , 1$, thus $H_1$ is smooth at $x$. 
	We see that $H_1$ is $\bQ$-Gorenstein, $L|_{H_1}$ is ample $\bQ$-Cartier, $\cO_{H_1}(kL|_{H_1})$ is CM for $k \in \bZ$ by Lemma \ref{lem:restrictexact}.      
	We also see that $H_1 \cap H_j \in |m_j L|_{H_1}|$ for $j=2, \ldots , n-1$    
	satisfy the conditions (i), (ii). 
	By the induction hypothesis, we obtain 
	\[
	\delta_x(L|_{H_1}) \ge \frac{n}{m_2 \cdots m_{n-1} (L|_{H_1})^{n-1} } = \frac{n}{m_1 \cdots m_{n-1} L^n}. 
	\]
	By Theorem \ref{thm:AZadjunctionineq}, we have 
	\[
	\delta_x(m_1L) \ge \min \left\{ n+1, \frac{n+1}{n} \delta_x(m_1L|_{H_1}) \right\}
	\] 
	and obtain the required inequality recalling that $\delta_x(L)=m\delta_x(mL)$.

Assume now that an equality holds and $\delta_x(L) \ne n+1$. Let $v$ be a valuation computing $\delta_x(L)$. 
By Theorem \ref{thm:AZadjunctionineq}, we see that $C_X(v) \not \subset H_1$ and $\delta_x(L|_{H_1})=\frac{n}{m_1 \cdots m_{n-1} L^n}$ is computed by some valuation $v_1$ on $H_1$ with the centre $Z \subset C_X(v) \cap H_1$.  By induction $v_1$ is divisorial and induced by a prime divisor on $H_1$, which implies that $C_X(v)$ is a divisor on $X$.

\end{proof}

\section{Main results} \label{s:results}

In this section we prove the main result Theorem \ref{thm:main_intro} (= Theorem \ref{thm:main}) and its corollaries.

\begin{remark}
	Let $X \subset \bP(a_0, \ldots , a_{n+1})$ be a weighted hypersurface with the defining polynomial $F$ of degree $d$
	and let $P_r = [0,\ldots , 0, 1, 0, \ldots , 0]$ be the $r$-th coordinate point. 
	We repeatedly use the fact that $P_r \notin X$ if and only if $F$ contains the monomial $z_r^{d/a_r}$ with a non-zero coefficient.    
\end{remark}

\begin{theorem}\label{thm:main}
	Let $X=X_d \subset \bP(a_0, \ldots, a_{n+1})$ be a well-formed quasi-smooth weighted hypersurface which is not a linear cone. 
	Assume that there is $r$  such that $a_r >1$ and $a_r | d$.
	Then 
	$$
	\delta(\cO_X(1)) \ge \frac{(n+1)a_{r}}{d}.
	$$
	
	Moreover, if  $X$ is Fano of index $\iota_X: = \sum_{i=0}^{n+1} a_i -d$ and   $\frac{(n+1)a_{r}}{d} \ge \iota_X$, then $\delta(-K_X) \ge 1$ and if $n \ge 3$, then $X$ is K-stable. 
\end{theorem}

\begin{proof}
By Lemma \ref{lem:P_r} we can assume that $P_r \notin X$.
Let $p=[p_0,\ldots, p_{n+1}]\in X$ be a point. 
We shall show that $\delta_p(\cO_X(1)) \ge \frac{(n+1)a_{r}}{d}$ in the following. 

Assume first that $a_n=1$, then $P_r=P_{n+1}$ and $X$ is smooth. By Lemma \ref{l:cuttingcurve}, we can take hyperplanes $H_j \in |\cO_X(1)|$ with $j = 2,\ldots,n$  so that $C= \cap_{j =2}^n H_j$ is a smooth curve passing through $p$ (in this case this is an easy observation). 
By Lemma \ref{l:curve} (or \cite[Lemma 4.4]{AZ22}), we get
$$
\delta_p(\cO_X(1)) \ge \frac{n+1}{\cO_X(1)^n} = \frac{(n+1)a_{n+1}}{d}
$$
since $\cO_X(1)^n=d/a_{n+1}$. If an equality holds, then $\delta_p(\cO_X(1))$ is computed by a divisor $E$ on $X$ by Lemma \ref{l:curve}. 
Assume now that $X$ is Fano with $\frac{(n+1)a_{n+1}}{d} \ge \iota_X$. If the inequality is strict, then we get $\delta_p(-K_X) >1$ directly.  
If  $\delta_p(\cO_X(1))= \frac{(n+1)a_{n+1}}{d} = \iota_X$ and $n \ge 3$, then we look at a prime divisor $E$ on $X$ which computes $\delta_p(\cO_X(1))$.  
Then we have $\beta_X(E):= A_{X}(E) - S(-K_X; E) =0$. 
Let $\tau E \sim -K_X$ (such $\tau$ exists, since $n \ge 3$ implies $\rho(X)=1$ (cf. \cite[Proposition 2.3]{PST17})). 
Note that since $a_{n+1}<d$, we have $\iota_X <n+1$ and so $\tau <n+1$.
Then we can compute $A_X(E)=1$ and $S(-K_X; E) = \frac{\tau}{n+1}$, thus we have $\beta_X(E) = 1- \frac{\tau}{n+1}$. 
Since $\tau < n+1$, we see that $\beta_X(E)>0$ and this is a contradiction. 
Hence $X$ is K-stable.

So we can assume $a_n >1$. The plan is to apply Lemma \ref{l:cuttingcurve}, but first we need to take one or two covers (depending whether $a_0=1$ or not) to meet the assumptions $a_0=a_1=1$ and $p \notin Bs|\cO_X(1)|$ of the Lemma \ref{l:cuttingcurve}.

\begin{claim}\label{c:p_k}
We can assume $p_k \ne 0$ for some $k \ne r$ such that $a_k >1$. 
\end{claim}
\begin{proof}[Proof of Claim \ref{c:p_k}]
If $p=[p_0,\ldots,p_{n+1}] \in Bs|\cO_X(1)|$, then there must be $p_k \ne 0$ with $k \ne r$ such that $a_k >1$, otherwise $p=P_r$, which is impossible since we are assuming that $P_r \notin X$. So we assume $p \notin Bs|\cO_X(1)|$. Up to reordering the coordinates with weights $1$ we can assume $p_0 \ne 0$. Applying an automorphim of the form $z_0 \mapsto z_0$ and $z_i \mapsto z_i - \lambda_i z_0^{a_i}$ with $\lambda_i \in \bC$ if $i>0$, we can assume that $p=[1,0,\ldots,0]$. Note that $P_r \notin X$ also after the automorphism.  
At this point we can send $p$ to $[1,0,\ldots,0,1,0,\ldots,0]$ where the second $1$ is at the $k$-th position for $k \ne r$ and $a_k >1$ (this is possible since $a_n >1$). Again $P_r \notin X$. This finishes the proof of the claim.	
\end{proof}

Set $H_k = \{ z_k=0 \} \subset X$ and note now that $H_k$ has only isolated singularities at the finite set
\[
Z_k:= \left\{ z_k= \frac{\partial F}{\partial z_0} = \cdots = \frac{\partial F}{\partial z_{k-1}}  = \frac{\partial F}{\partial z_{k+1}} = \cdots =  \frac{\partial F}{\partial z_{n+1}} =0 \right\} \cap X, 
\]  
where $F$ is the defining polynomial of $X$. This is because otherwise $Z_k \cap \{ \frac{\partial F}{\partial z_{k}} =0\}$ would be non-empty, contradicting the quasi-smoothness of $X$.
Let $\delta_k \subset |\cO_X(a_k)|$ be the sublinear system given by $z_0^{a_k}, \ldots , z_{c_1-1}^{a_k}, z_k$. 
Since we have 
\[
Bs (\delta_k) \subset Bs |\cO_X(1)| \cap \{ z_k =0 \}=:L_k,
\]
 we can take a general homogeneous polynomial $G_k \in \bC[z_0, \ldots , z_{c_1-1}]$ of degree $a_k$ 
so that $H'_k= \{ z_k + G_k =0 \} \in |\cO_X(a_k)|$  is quasi-smooth outside $L_k$. If $a_0 >1$, we let $Bs |\cO_X(1)|:=X$ 
and let $H'_k := H_k$ without using $G_k$. 
We see that $p \notin H'_k$ since $p_k \neq 0$ in Claim \ref{c:p_k} and $G_k$ is general when $a_0 =1$.

Consider the morphism 
\[
\pi \colon \bP':=\bP(a_0,\ldots,a_{k-1},1,a_{k+1},\ldots,a_{n+1}) \to \bP(a_0,\ldots,a_k, \ldots , a_{n+1})=:\bP
\] given by $w_i \mapsto w_i=z_i$ for $i \ne k$ and $w_k \mapsto w_k^{a_k} - G_k=z_k$. 

Let $W = \pi^{-1}(X)$ and note that $\cO_W(1) = \pi^*\cO_{X}(1)$. The variety $W$ is a well-formed hypersurface of degree $d$ defined by the polynomial 
\[
G(w_0,\ldots,w_{n+1})=F(w_0, \ldots,w_{k-1}, w_k^{a_k}-G_k, w_{k+1}, \ldots, w_{n+1}) \in \bC[w_0, \ldots, w_{n+1}]. 
\] 

Since $\pi$ is a finite cyclic cover branched along $H'_k= \{z_k+G_k =0\}$ whose singular points are contained in $L_k$, 
we see that $W$ has only isolated hypersurface singularities (thus normal) 
and $\NQS(W) \subset Bs|\cO_W(1)|$ since $\pi^{-1}(L_k) = Bs |\cO_W(1)|$.
Let $q=[q_0, \ldots, q_{n+1}] \in W$ be a preimage of $p$, so that $q_k \ne 0$ by $p \notin H'_k$. 

Since we have 
\[
K_W = \pi^* \left( K_X + \frac{a_k-1}{a_k} H'_k \right)  
\]
and the pair $\dps{ \left( X, \Delta_k:=\frac{a_k-1}{a_k} H'_k \right)}$ is klt at $p$ by $p \notin H'_k$, we see that $W$ is klt at $q$. 
By $p \notin H'_k$ and the definition of $\delta_p$, we have  
\[
\delta_p(X, \Delta_k; \cO_X(1)) = \delta_p(X, 0; \cO_X(1)) = \delta_p(\cO_X(1)). 
\]
By this and Proposition \ref{prop:deltafinitecover}, we have 
\[
\delta_q(\cO_W(1)) = \delta_q(W,0; \cO_W(1)) \le \delta_p(\cO_X(1)).
\]  

Note that $q \notin Bs|\cO_W(1)|$ since the weight of $w_k$ is 1 and we still have $P'_r \notin W$, 
where $P'_r \in \bP'$ is also the $r$-th coordinate point. 

We now distinguish two cases.

\noindent\textbf{Case 1.} Assume $a_0=1$.  Up to reordering the coordinates we can move $w_k$ to $w_1$, so we have $W \subset \bP(1,1,b_2,\ldots,b_{n+1})$,  where 
$b_j = \begin{cases} 
a_{j-1} & (j \le k) \\
a_j & (j > k)
\end{cases}$. 

Note that if $k \ge r$, we should change $r$ to $r+1$, but for simplicity we stick to use the same letter $r$ and assume $P'_r \notin W$ 
since this creates no problem. 
(If $k=n+1$ and $r=n$, then $r$ becomes $n+1$ after reordering. This is not a problem since we have $M = \emptyset$ in Case 1 to apply Lemma 4.3.)

By Lemma \ref{l:cuttingcurve} we can take hyperplanes $H_j \in |\cO_W(b_j)|$ for $j \in J:=\{2,\ldots,n+1\}\setminus \{r\}$ so that $C= \cap_{j \in J }H_j$ is a smooth curve passing through $q$.
By Lemma \ref{l:curve}  we get  
\[
\delta_q( \cO_W(1)) \ge \frac{n+1}{(\prod_{j \in J}b_j) \cdot \cO_W(1)^n} = \frac{n+1}{(\prod_{j \in \{1,\ldots , n+1 \} \setminus \{k,r \} } a_j) \cO_W(1)^n} = \frac{(n+1)a_{r}}{d}.
\]
since $\cO_W(1)^n = a_k \cO_X(1)^n= \dps{ a_k \cdot \frac{d}{\prod_{i=1}^{n+1} a_i}}$ (and $a_0=1$). 

Assume now that $X$ is Fano of index $\iota_X \le \frac{(n+1)a_r}{d}$. 
Since we have $\delta_p(\cO_X(1)) \ge \delta_q(\cO_W(1)) \ge \frac{(n+1)a_r}{d}$, we obtain 
\[
\delta_p(-K_X) = \delta_p(\cO_X(\iota_X)) = \frac{1}{\iota_X} \delta_p(\cO_X(1)) \ge \frac{1}{\iota_X} \cdot \frac{(n+1)a_r}{d} \ge 1. 
\] 
If 
$
\delta_p(-K_X)
= 1,
$
then we have $\iota_X = \frac{(n+1)a_r}{d}$. 
By \cite[Theorem 1.2]{LXZ22}, $\delta_p(-K_X)$ is computed by a divisorial valuation over $X$. 
Moreover, we must have the equalities
\[
\delta_p(\cO_X(1)) =\delta_q(\cO_W(1)) = \frac{(n+1)a_r}{d}.
\]  
Hence by Lemma \ref{l:curve} every valuation on $W$ computing $\delta_q(\cO_W(1))$ is induced by a prime divisor on $W$. By Proposition \ref{prop:deltafinitecover}, $\delta_p(\cO_X(1))$ is computed by a prime divisor $E$ on $X$ and so $\beta_X(E)=0$, which gives a contradiction as in the case $a_n=1$.   

\smallskip 

\noindent\textbf{Case 2.} Assume  $a_0>1$. After taking the (first) cover $\pi$ branched along $H_k$ (and reordering the coordinates), we get a  well-formed hypersurface $W \subset \bP(1,b_1,b_2,\ldots,b_{n+1})$ of degree $d$ such that 
\[
\NQS(W) \subset Bs |\cO_W(1)|=\{w_0=0\} \subset W
\] is a finite set, $q=[q_0,\ldots,q_{n+1}] \notin Bs |\cO_W(1)|$ (that is, $q_0 \neq 0$) and $P'_r \notin W \subset \bP'$ (this $r$ is a new $r$ after the reordering).

As in Claim \ref{c:p_k}, we can assume $q_\ell \ne 0$ for some $\ell \ne r$ such that $b_\ell >1$. 
Consider $H'_\ell \in |\cO_W(b_\ell)|$ general. Since 
$$
Bs |\cO_W(b_\ell)| \subset Bs |\cO_W(1)| \cap (w_l=0)  =(w_0 = w_l=0), 
$$
we conclude by Bertini that $H'_\ell$ is quasi-smooth outside 
$(w_0 = w_l =0) \cup \NQS(W)$ (note that $H'_{\ell}$ can be singular along a curve).

Hence we can take the cover $\nu \colon Y \to W$ branched along $H'_\ell$ and obtain (up to reordering) a variety isomorphic to a well-formed hypersurface 
\[
Y \subset \bP(1,1,c_2,\ldots,c_{n+1})
\]
such that $\NQS(Y) \subset \nu^{-1} \left( (w_0=w_\ell=0) \cup \NQS(W) \right) $, where $Bs|\cO_Y(1)|$ is the inverse image of $(w_0 = w_\ell =0)$. 
Let $y_0,\ldots, y_{n+1}$ be the coordinates of $\bP(1,1,c_2,\ldots,c_{n+1})$.
We now check the normality of $Y$. Since $\dim \NQS(Y) \le 1$, the only problem is if $n=2$. In this case, we have $\NQS(Y) \ne Bs |\cO_Y(1)|$, because the $r$-th coordinate point $P''_r \notin Y$ (again a new $r$). Hence the codimension of $\NQS(Y)$ is at least two and $Y$ is normal. 
For a preimage $q' \in \nu^{-1}(q)$, we see that $Y$ is klt at $q'$ by $q \notin H'_{\ell}$ as before. 
We also have the inequality $\delta_{q'}(Y, \cO_Y(1)) \le \delta_q(W, \cO_W(1))$ as before. 

Note also that $M= \nu^{-1} (\NQS(W))$ is a finite set which is contained in $(y_0=0)$ and that 
$q' \notin (y_0 =0)$ by the construction of $\nu$. 
By Lemma \ref{l:cuttingcurve}, we can take hyperplanes $H_j \in |\cO_Y(c_j) \otimes \cI_{q'}|$ with the properties in Lemma \ref{l:cuttingcurve}. 
We can now proceed to show $\delta_p(\cO_X(1)) \ge \delta_{q'}(\cO_Y(1)) \ge \frac{(n+1)a_r}{d}$ exactly as in the proof of \textbf{Case 1}, noting that $\cO_Y(1)^n=a_k a_\ell \cO_X(1)^n$ and $\cO_X(1)^n=d/\prod_{i=0}^{n+1} a_i$. 

The K-stability of $X$ when $X$ is Fano  of index $\iota_X \le \frac{(n+1) a_r}{d}$ follows from the exact same argument as in \textbf{Case 1}.
\end{proof}

\begin{proof}[Proof of Corollary \ref{c:quasi-smooth}]
	By Theorem \ref{thm:main}, we have 
	$$
	\delta(\cO_X(1)) \ge \frac{(n+1)a_{n+1}}{d}.
	$$
	The result follows now from Lemma \ref{l:bound1} and, to treat the equality case, we again use Theorem \ref{thm:main}.
\end{proof}

\begin{proof}[Proof of Corollary \ref{c:smooth}]
By \cite[Corollary 5.11]{PST17} we know that $a_0=a_1=1$. In addition, since $X$ is smooth, we also have that $P_{n+1}=[0,\ldots,0,1] \notin X$. 
By Theorem \ref{thm:main}, we have 
$$
\delta(\cO_X(1)) \ge \frac{(n+1)a_{n+1}}{d}.
$$
The result follows now from Proposition \ref{p:bound}, where the cases $\gamma=1$ are taken care by Theorem \ref{thm:main}. 


\end{proof}

\section*{Acknowledgement} 
We thank Kento Fujita, Yuchen Liu and Takuzo Okada for valuable comments. 
We thank Victor Przyjalkowski for valuable communications on Subsection \ref{ss:WCI}. 

Part of this project has been realized when the first author visited Universit\`a degli Studi di Milano. He would like to thank the university and Kobe university for their support. 
The first author was partially supported by JSPS KAKENHI Grant Numbers JP17H06127, JP19K14509, JP23K03032 and Long-term Overseas
Teaching Fellowship Program for Young Researchers from Kobe university. 
The second author was partially supported by PRIN2020 research grant ”2020KKWT53” and  is member of the GNSAGA group of INdAM.

\bibliographystyle{amsalpha}
\bibliography{Library}

\end{document}

%% file: macros.tex
\usepackage{amsfonts,amsmath,amsthm,amssymb,latexsym,amsthm,newlfont,enumerate,color}

\newif\ifslide

\theoremstyle{plain}

\ifdefined\spacer
\newtheorem{theorem}{Theorem}
 \else
\newtheorem{theorem}{Theorem}[section]
\fi

\newtheorem{corollary}[theorem]{Corollary}
\newtheorem{lemma}[theorem]{Lemma}
\newtheorem{claim}[theorem]{Claim}

\newtheorem*{theorem*}{Theorem}

\newtheorem{proposition}[theorem]{Proposition}
\newtheorem{definition-lemma}[theorem]{Definition-Lemma}
\newtheorem{question}[theorem]{Question}
\newtheorem{red-question}[theorem]{\textcolor{red}{Question}}

\theoremstyle{definition}
\newtheorem{definition}[theorem]{Definition}
\newtheorem{remark}[theorem]{Remark}
\newtheorem{setup}[theorem]{Setup}
\newtheorem{example}[theorem]{Example}

\def\ol#1{\overline{#1}}

\def\ideal#1.{I_{#1}}
\def\ring#1.{\mathcal {O}_{#1}}
\def\Spec{\operatorname {Spec}}

\def\fring#1.{\hat{\mathcal {O}}_{#1}}
\def\proj#1.{\mathbb {P}(#1)}
\def\pr #1.{\mathbb {P}^{#1}}
\def\dpr #1.{\hat{\mathbb {P}}^{#1}}
\def\af #1.{\mathbb A^{#1}}
\def\Hz #1.{\mathbb F_{#1}}
\def\Hbz #1.{\overline{\mathbb F}_{#1}}
\def\fb#1.{\underset #1 {\times}}
\def\rest#1.{\underset {\ \ring #1.} \to \otimes}
\def\au#1.{\operatorname {Aut}\,(#1)}
\def\deg#1.{\operatorname {deg } (#1)}

\def\pic#1.{\operatorname {Pic}\,(#1)}
\def\pico#1.{\operatorname{Pic}^0(#1)}
\def\picg#1.{\operatorname {Pic}^G(#1)}
\def\ner#1.{NS (#1)}
\def\rdown#1.{\llcorner#1\lrcorner}
\def\rfdown#1.{\lfloor{#1}\rfloor}
\def\rup#1.{\ulcorner{#1}\urcorner}
\def\rcup#1.{\lceil{#1}\rceil}

\def\n1#1.{\operatorname {N_1}(#1)}  
\def\cn1#1.{\overline{\operatorname {N^1}(#1)}} 
\def\cone#1.{\operatorname {NE}(#1)}     
\def\ccone#1.{\overline{\operatorname {NE}}(#1)}
\def\none#1.{\operatorname {NF}(#1)}
\def\cnone#1.{\overline{\operatorname {NF}}(#1)}
\def\mone#1.{\operatorname {NM}(#1)} 
\def\cmone#1.{\overline{\operatorname {NM}}(#1)}

\def\coef#1.{\frac{(#1-1)}{#1}}
\def\vit#1.{D_{\langle #1 \rangle}}
\def\mm#1.{\overline {M}_{0,#1}}
\def\H1#1.{H^1(#1,{\ring #1.})}
\def\ac#1.{\overline {\mathbb F}_{#1}}

\def\adj#1.{\frac {#1-1}{#1}}
\def\spn#1.{\overline{#1}}
\def\pek#1.#2.{\Cal P^{#1}(#2)}
\def\plk#1.#2.{\Cal P^{\leq #1}(#2)}
\def\ev#1.{\operatorname{ev_{#1}}}
\def\ilist#1.{{#1}_1,{#1}_2,\dots}
\def\bminv#1.{(\nu_1,s_1;\nu_2,s_2;\dots ;\nu_{#1},s_{#1};\nu_{r+1})}
\def\zinv#1.{(\nu_1,s_1;\nu_2,s_2;\dots ;\nu_{#1},s_{#1};0)}
\def\iinv#1.{(\nu_1,s_1;\nu_2,s_2;\dots ;\nu_{#1},s_{#1};\infty)}

\def\scr #1.{\mathcal #1}

\def\llist#1.#2.{{#1}_1,{#1}_2,\dots,{#1}_{#2}}
\def\ulist#1.#2.{{#1}^1,{#1}^2,\dots,{#1}^{#2}}
\def\lomitlist#1.#2.{{#1}_1,{#1}_2,\dots,\hat {{#1}_i}, \dots, {#1}_{#2}}
\def\lomitlistz#1.#2.{{#1}_0,{#1}_1,\dots,\hat {{#1}_i}, \dots, {#1}_{#2}}
\def\loc#1.#2.{\Cal O_{#1,#2}}
\def\fderiv#1.#2.{\frac {\partial #1}{\partial #2}}
\def\deriv#1.#2.{\frac {d #1}{d #2}}

\def\map#1.#2.{#1 \longrightarrow #2}
\def\rmap#1.#2.{#1 \dasharrow #2}
\def\emb#1.#2.{#1 \hookrightarrow #2}
\def\non#1.#2.{\text {Spec }#1[\epsilon]/(\epsilon)^{#2}}
\def\Hi#1.#2.{\text {Hilb}^{#1}(#2)}
\def\sym#1.#2.{\operatorname {Sym}^{#1}(#2)}
\def\Hb#1.#2.{\text {Hilb}_{#1}(#2)}
\def\Hm#1.#2.{\Hom_{#1}(#2)}
\def\prd#1.#2.{{#1}_1\cdot {#1}_2\cdots {#1}_{#2}}
\def\Bl #1.#2.{\operatorname {Bl}_{#1}#2}
\def\pl #1.#2.{#1^{\otimes #2}}
\def\mgn#1.#2.{\overline {M}_{#1,#2}}
\def\ialist#1.#2.{{#1}_1 #2 {#1}_2, #2\dots}
\def\pair#1.#2.{\langle #1, #2\rangle}
\def\vandermonde#1.#2.{\left|
\begin{matrix}
1 & 1 & 1 & \dots & 1\\
{#1}_1 & {#1}_2 & {#1}_3 & \dots & {#1}_{#2}\\
{#1}_1^2 & {#1}_2^2 & {#1}_3^2 & \dots & {#1}_{#2}^2\\
\vdots & \vdots & \vdots & \ddots & \vdots\\
{#1}_1^{#2-1} & {#1}_2^{#2-1} & {#1}_2^{#2-1} & \dots & {#1}_{#2}^{#2-1}\\
\end{matrix}
\right|
}
\def\vandermondet#1.#2.{\left|
\begin{matrix}
1 & {#1}_1   & {#1}_1^2 & \dots & {#1}_1^{#2-1}\\
1 & {#1}_2   & {#1}_2^2 & \dots & {#1}_2^{#2-1}\\
1 & {#1}_3   & {#1}_3^2 & \dots & {#1}_3^{#2-1}\\
\vdots & \vdots & \vdots & \ddots & \vdots\\
1 & {#1}_{#2}& {#1}_{#2}^2 & \dots & {#1}_{#2}^{#2-1}\\
\end{matrix}
\right|
}
\def\gr#1.#2.{\mathbb{G}(#1,#2)}

\def\alist#1.#2.#3.{{#1}_1 #2 {#1}_2 #2\dots #2 {#1}_{#3}}
\def\zlist#1.#2.#3.{#1_0 #2 #1_1 #2\dots #2 #1_{#3}}
\def\lomitlist30#1.#2.#3.{{#1}_0,{#1}_1 #2 \dots #2\hat {{#1}_i} #2\dots #2 {#1}_{#3}}
\def\lmap#1.#2.#3.{#1 \overset{#2}{\longrightarrow} #3}
\def\mes#1.#2.#3.{#1 \longrightarrow #2 \longrightarrow #3}
\def\ses#1.#2.#3.{0\longrightarrow #1 \longrightarrow #2 \longrightarrow #3 \longrightarrow 0}
\def\les#1.#2.#3.{0\longrightarrow #1 \longrightarrow #2 \longrightarrow #3}
\def\res#1.#2.#3.{#1 \longrightarrow #2 \longrightarrow #3\longrightarrow 0}
\def\Hi#1.#2.#3.{\text {Hilb}^{#1}_{#2}(#3)}
\def\ten#1.#2.#3.{#1\underset {#2}{\otimes} #3}
\def\lomitlist30#1.#2.#3.{{#1}_0 #2 {#1}_1 #2 \dots #2 \hat {{#1}_i} #2 \dots #2 {#1}_{#3}}
\def\mderiv#1.#2.#3.{\frac {d^{#3} #1}{d #2^{#3}}}

\def\Hom{\operatorname{Hom}}

\def\Proj{\operatorname{Proj}}

\def\Supp{\operatorname{Supp}}

\def\dim{\operatorname{dim}}

\def\deg{\operatorname{deg}}
\def\depth{\operatorname{depth}}

\def\Sing{\operatorname{Sing}}

\def\ord{\operatorname{ord}}

\def\rest{\operatorname{res}}

\def\vol{\operatorname{vol}}

\def\e{\Cal E}

\def\e1{E_1}
\def\e2{E_2}

\def\mapdown#1{\big\downarrow\rlap{$\vcenter{\hbox{$\scriptstyle#1$}}$}}

\def\mapse#1{
{\vcenter{\hbox{$\mathop{\smash{\raise1pt\hbox{$\diagdown$}\!\lower7pt
\hbox{$\searrow$}}\vphantom{p}}\limits_{#1}\vphantom{\mapdown{}}$}}}}

\def\VR#1.{height#1pt&\omit&&\omit&&\omit&&\omit&&\omit&\cr}

\def\VRT#1.{height#1pt&\omit&&\omit&\cr}
